\documentclass[12pt]{article}
\usepackage{amsmath,amsthm,amssymb}

\usepackage{mathrsfs}

\allowdisplaybreaks

\newtheorem{theorem}{Theorem}
\newtheorem{corollary}{Corollary}
\newtheorem{lemma}{Lemma}
\newtheorem{proposition}{Proposition}
\theoremstyle{remark}
\theoremstyle{remark}
\theoremstyle{remark}
\newtheorem{remark}{Remark}

\newcommand{\di}{\partial}
\newcommand{\rom}[1]{{\rm #1}}

\newcommand{\R}{\mathbb{R}}

\newcommand{\N}{\mathbb{N}}
\newcommand{\eps}{\varepsilon}
\newcommand{\la}{\langle}
\newcommand{\ra}{\rangle}
\newcommand{\FC}{{\cal F}C_{\mathrm b}(C_0(\mathbb R^d),\Gamma)}
\newcommand{\FCC}{{\cal F}C^\infty_{\mathrm b}(C^\infty_0(\mathbb R^d),\Gamma)}

\begin{document}

\makeatletter\@addtoreset{equation}{section}

\begin{center}{\Large \bf
  Binary jumps in continuum. I. Equilibrium processes and their scaling limits
}\end{center}

{\large Dmitri  L. Finkelshtein}\\
Institute of Mathematics, National Academy of Sciences of Ukraine, 3 Tereshchenkivska Str.,
Kyiv 01601, Ukraine.\\
e.mail: \texttt{fdl@imath.kiev.ua}\vspace{2mm}

{\large Yuri G. Kondratiev}\\
 Fakult\"at f\"ur Mathematik, Universit\"at
Bielefeld, Postfach 10 01 31, D-33501 Bielefeld, Germany;  NPU, Kyiv, Ukraine\\
 e-mail:
\texttt{kondrat@mathematik.uni-bielefeld.de}\vspace{2mm}

{\large Oleksandr V.  Kutoviy}\\
 Fakult\"at f\"ur Mathematik, Universit\"at Bielefeld, Postfach 10
01 31, D-33501 Bielefeld, Germany; National Dragomanov University,
Pirogova Str. 9,
Kyiv 01601,
Ukraine
e.mail: \texttt{kutoviy@mathematik.uni-bielefeld.de}\vspace{2mm}

{\large Eugene Lytvynov}\\
Department of Mathematics,  Swansea University, Singleton
Park, Swansea SA2 8PP, U.K.\\ e-mail:
\texttt{e.lytvynov@swansea.ac.uk}

{\small
\begin{center}
{\bf Abstract}
\end{center}
\noindent Let $\Gamma$ denote the space
of all locally finite subsets (configurations) in $\R^d$. A stochastic dynamics of binary jumps in continuum is a Markov process on $\Gamma$ in which
pairs of particles simultaneously hop over $\R^d$. In this paper, we study an equilibrium dynamics of binary jumps for which a Poisson measure is a symmetrizing (and hence invariant) measure.  The existence and uniqueness of the corresponding stochastic  dynamics are shown. We next prove the main result of this paper: a big class of dynamics of binary jumps converge, in a diffusive scaling limit, to a dynamics of interacting Brownian particles. We also study another scaling limit, which leads us to a spatial birth-and-death process in continuum. A remarkable property of the limiting dynamics is that  its generator possesses a spectral gap, a property which is hopeless to expect from the initial dynamics of binary jumps.

 }
\vspace{3mm}

\vspace{1.5mm}

\vspace{1.5mm}


\section{Introduction} Let $\Gamma=\Gamma_{\R^d}$ denote the space
of all locally finite subsets (configurations) in $\R^d$, $d\in\mathbb N$.  A stochastic dynamics of binary jumps in continuum is a Markov process on $\Gamma$ in which
pairs of particles simultaneously hop over $\R^d$, i.e., at each jump time two points of the configuration change their position.
Thus, an (informal) generator of such a process has the form
\begin{multline} (LF)(\gamma)=\sum_{\{x_1,\,x_2\}\subset\gamma}\int_{(\R^d)^2}Q( x_1,x_2,dh_1\times dh_2)\\
\times
\big(F(\gamma\setminus\{x_1,x_2\}\cup\{x_1+h_1,x_2+h_2\})-F(\gamma)\big).\label{tyrsar5a}\end{multline}
Here, the measure  $Q( x_1,x_2,dh_1\times dh_2)$
describes the rate at which two particles, $x_1$ and $x_2$, of configuration $\gamma$ simultaneously hop to $x_1+h_1$ and $x_2+h_2$, respectively. Generally speaking, this rate may also depend on the rest of the configuration,\linebreak $\gamma\setminus\{x_1,x_2\}$. However, in our current studies, we will
restrict out attention to the case  of the generator  \eqref{tyrsar5a} only. As the reader will see below, already such  dynamics lead, in a  scaling limit, to interesting diffusion dynamics.

The stochastic dynamics of binary jumps may be compared with the Kawasaki dynamics in continuum. The latter is a Markov process on $\Gamma$ in which particles hop over $\R^d$ so that, at each jump time, only one particle changes its position. For a study of equilibrium  Kawasaki dynamics in continuum, we refer the reader to the papers \cite{FKL,FKO,G2, KKL,KLR, LO} and the references therein.

In this paper, we will study an equilibrium dynamics of binary jumps for which a Poisson measure is a symmetrizing (and hence invariant) measure.
In several cases, an equilibrium stochastic dynamics on $\Gamma$  with a Poisson symmetrizing measure  is a free dynamics, i.e., there is no interaction between particles. For example,  this is  true for a Surgailis process (in particular, the Glauber dynamics without interaction)
 \cite{Surgailis,Surgailis2} (see also \cite{KLR_free}). Note that a Surgailis  generator, in the symmetric Fock space realization of the $L^2$-space of   Poisson measure, is the second quantization of the generator of a one-particle dynamics. Another example of a Surgalis dynamics   is the free Kawasaki dynamics. There a  Poisson measure is a symmetrizing one, and in the course of   random evolution   each particle of the configuration randomly hops over $\mathbb R^d$ without any interaction with other particles.

Let us stress at this point one essential difference between lattice and continuous systems. An important example of a Markov dynamics
on lattice configurations is the so-called exclusion process. In this process particles jump over the lattice with only restriction to have no more
than one particle at each  point of the lattice. This process has a Bernoulli measure as an invariant  (and even   symmetrizing) measure
but the corresponding stochastic dynamics has non-trivial properties and possess  interesting and reach scaling limit behaviors.  A straightforward
generalization of the exclusion process  to the continuum  gives just free Kawasaki dynamics because the exclusion restriction
(and  an interaction between particles) will
obviously disappear for configurations in continuum.
To introduce (in certain sense simplest) interaction we consider the generator   above.
The dynamics
of binary jumps is not anymore a free particle process. In fact, in the  mentioned Fock space realization, the generator of this dynamics has a Jacobi matrix (three-diagonal) form.

The reader may wish to compare our results to the paper \cite{AGY}, which contains a discussion on random fields with interaction which are induced by Poisson random fields. 

When this paper was nearing completion, the reference    \cite{BMT}    came to our attention. There, at a rather heuristic level, the authors discuss a
special kind of a stochastic dynamics of binary jumps, and derive a
Boltzmann-type equation through a Vlasov-type scaling limit of this dynamics. In particular, the underlying space  is $\mathbb R^3$, and points in $\mathbb R^3$ are treated as velocities of particles, rather than their positions.  As a result of pair interaction, at  random times,
two particles change their velocities from  $ v_i$ and $ v_j$ to
 $ v'_i= v_i+ h$ and $ v_j'= v_j- h$, respectively. Thus,  $ v_i+ v_j= v'_i+ v'_j$ and hence the law of conservation of momentum is satisfied for this system. Hence, for such a dynamics, the measure $Q(x_1,x_2,dh_1\times dh_2)$ in formula \eqref{tyrsar5a} is concentrated on the set
 $$\{(h,-h)\mid h\in\mathbb R^3\}\subset (\R^3)^2. $$
In fact, further assumptions on  $Q$ appearing in \cite{BMT} are almost identical to ours in this special case. Throughout the paper, we have added a series of statements and remarks regarding such a dynamics.

 The paper is organized as follows.
 In Section \ref{lkgfty}, using the theory of Dirichlet forms \cite{Fu80,MR}, we
construct a rather general dynamics on the configuration space, whose generator has the form
 \eqref{tyrsar5a} on a set of test cylinder functions on $\Gamma$.

 In Section \ref{gufr75eses}, we show that the generator \eqref{tyrsar5a} with domain being the set of test cylinder functions uniquely identifies a Markov process on $\Gamma$. More exactly, this generator is essentially self-adjoint in the $L^2$-space of  Poisson measure. The proof of the essential self-adjointness 
 is done through an explicit formula for the form of the generator \eqref{tyrsar5a} realized as an operator acting in the symmetric Fock space (compare with e.g.\ \cite{AFY} and the references therein). The reader may find this formula for the generator in the Fock space to be of independent interest.

The central result of the paper is in Section \ref{vcydryds}, where we show that a big class of dynamics of binary jumps converge, in a diffusive scaling limit, to a dynamics of interacting Brownian particles.  The form of the  generator of the limiting diffusion resembles the generator
  of the gradient stochastic dynamics (e.g.\  \cite{AKR,Osada,Spohn,Yoshida}), while staying symmetric with respect to the Poisson measure, rather than with respect to a Gibbs measure (as it is the case for the gradient stochastic dynamics). We prove the convergence of processes at the level of convergence, in the $L^2$-norm, of their generators applied to a test cylinder function on $\Gamma$.

 Finally, in Section \ref{ufdey76ersr}, we study another scaling limit of a class of dynamics of
 binary jumps which leads to a spatial birth-and-death process in continuum, in which pairs of particles, as well as single particles randomly appear (are born) and disappear (die). We prove the convergence of processes at the level of weak convergence of their finite-dimensional distributions.  A remarkable property of the limiting dynamics is that its generator possesses  a spectral gap, a property which is hopeless to expect from the generator of the initial dynamics of binary jumps. We also note that the result of the scaling essentially depends on the initial distribution of the dynamics.

In the second part of this paper  \cite{FKKL} we discuss non-equilibrium dynamics of binary jumps. In particular, we show that
a Vlasov-type mesoscopic scaling for such a dynamics leads to
a generalized Boltzmann non-linear equation for the particle density.

 \section{Existence of dynamics}\label{lkgfty}

The configuration space over $\R^d$, $d\in\N$,
is defined as the set of all subsets of $\R^d$ which are locally
finite: $$\Gamma:=\big\{\,\gamma\subset \R^d\mid
|\gamma_\Lambda|<\infty\text{ for each  compact }\Lambda\subset\R^d\,\big\}.$$
Here $|\cdot|$ denotes the cardinality of a set and
$\gamma_\Lambda:= \gamma\cap\Lambda$. One can identify any
$\gamma\in\Gamma$ with the positive Radon measure
$\sum_{x\in\gamma}\delta_x\in{\cal M}(\R^d)$, where  $\delta_x$ is
the Dirac measure with mass at $x$, and  ${\cal M}(\R^d)$
 stands for the set of all
positive  Radon  measures on the Borel $\sigma$-algebra ${\cal
B}(\R^d)$. The space $\Gamma$ can be endowed with the relative
topology as a subset of the space ${\cal M}(\R^d)$ with the vague
topology, i.e., the weakest topology on $\Gamma$ with respect to
which  all maps $$\Gamma\ni\gamma\mapsto\la f,\gamma\ra:=\int_{\R^d}
f(x)\,\gamma(dx) =\sum_{x\in\gamma}f(x),\quad f\in C_0(\R^d),$$ are
continuous. Here, $C_0(\R^d)$ is the space of all
continuous  functions on $\R^d$ with
compact support. We will denote by ${\cal B}(\Gamma)$ the Borel
$\sigma$-algebra on $\Gamma$.

We introduce the set $\FC$
of all functions on $\Gamma$ of the form \begin{equation}\label{drtses} F(\gamma)=g(\la\varphi_1,\gamma\ra,\dots , \la\varphi_N,\gamma\ra), \end{equation} where $N\in\N$, $\varphi_1,\dots,\varphi_N\in C_0(\mathbb R^d)$ and $g\in C_{\mathrm b}(\R^N)$,
where $C_{\mathrm b}(\R^N)$ denotes the space of all continuous bounded functions on $\R^N$.

For any $x_1,x_2\in\R^d$, $x_1\ne x_2$, let $Q(x_1,x_2,dh_1\times dh_2)$ be a measure on $\big((\R^d)^2,\mathcal B ((\R^d)^2)\big)$. Some assumptions on $Q$ will be discussed below.
We are interested in a (formal) pre-generator  of a Markov processes on $\Gamma$ which has the form
\eqref{tyrsar5a} on the set $\FC$.
 We assume that, for any fixed $A,B\in\mathcal B(\R^d)$, $$(x_1,x_2)\mapsto Q(x_1,x_2,A\times B)$$ is a measurable function.
In order that the integration in \eqref{tyrsar5a} do not depend on the order of $x_1,x_2$, we also assume that
\begin{equation}\label{dtsrea}Q(x_1,x_2,A\times B)=Q(x_2,x_1,B\times A),\quad x_1,x_2\in\R^d,\ x_1\ne x_2,\ A,B\in \mathcal B(\R^d).\end{equation}

We would like $(L,\FC)$ to be a symmetric operator in the (real) $L^2$-space $L^2(\Gamma,\pi_z)$. Here $\pi_z$ denotes the Poisson measure on $(\Gamma,\mathcal B(\Gamma))$ with intensity measure $z\,dx$, $z>0$. We recall that $\pi_z$ is uniquely characterized by the Mecke identity: for any measurable function $G:\Gamma\times\R^d\to[0,\infty]$,
\begin{equation}\label{Mecke}\int_\Gamma\pi_z(d\gamma)\sum_{x\in\gamma} G(\gamma,x)=\int_\Gamma\pi_z(d\gamma)\int_{\R^d}z\,dx\, G(\gamma\cup\{x\},x).\end{equation}

Assume that, for each  $\Lambda\in\mathcal B_0(\R^d)$ (a bounded Borel subset of $\R^d$),
\begin{equation}\label{hdtsrr}
\int_{\R^d}dx_1\int_{\R^d}dx_2 \int_{(\R^d)^2}Q(x_1,x_2,dh_1\times dh_2)(\mathbf 1_\Lambda(x_1)+\mathbf 1_\Lambda(x_1+h_1))<\infty.\end{equation}
Here,  $\mathbf 1_\Lambda$ denotes the indicator function of $\Lambda$. Using \eqref{tyrsar5a} and \eqref{dtsrea}--\eqref{hdtsrr},
one easily concludes that the quadratic form
$$\mathcal E(F,G):=\int_\Gamma (-LF)(\gamma)G(\gamma)\,\pi_z(d\gamma),\quad F,G\in\FC,$$
is well defined.

In order to achieve the symmetry of $\mathcal E$, we will assume that there exists a measure $m$ on $\big((\R^d)^2,\mathcal B ((\R^d)^2)\big)$ such that
\begin{equation}\label{xsreasr}Q(x_1,x_2,dh_1\times dh_2)=m(dh_1\times dh_2)\,q(x_1,x_2,h_1,h_2),\end{equation}
where $q:(\R^d)^4\to[0,\infty]$ is a measurable function.

\begin{lemma}\label{vtrsdty}
Assume that, for any $x_1,x_2\in\mathbb R^d$, $x_1\ne x_2$, we have the equality of measures
\begin{equation}\label{cfsres}
m(dh_1\times dh_2)q(x_1,x_2,h_1,h_2)=m'(dh_1\times dh_2)q(x_1+h_1,x_2+h_2,-h_1,-h_2),
\end{equation}
where $m'$ denotes the pushforward of the measure  $m$ under the mapping $(h_1,h_2)\mapsto(-h_1,-h_2)$.
Then, for any $F,G\in\FC$,
\begin{multline}\mathcal E (F,G)=\frac12\int_\Gamma\pi_z(d\gamma)\sum_{\{x_1,\,x_2\}\subset\gamma}
\int_{(\mathbb R^d)^2}m(dh_1\times dh_2)q(x_1,x_2,h_1,h_2)\\ \times
\big(F(\gamma\setminus\{x_1,x_2\}\cup\{
x_1+h_1,x_2+h_2\})-F(\gamma)\big) \big(G(\gamma\setminus\{x_1,x_2\}\cup\{
x_1+h_1,x_2+h_2\})-G(\gamma)\big)\label{hcfthg}.\end{multline}
In particular, the quadratic form $(\mathcal E,\FC)$ is symmetric.
\end{lemma}

\begin{proof}
Using the Mecke identity \eqref{Mecke} and formula \eqref{cfsres}, we have
\begin{align*}
&\int_\Gamma\pi_z(d\gamma)\sum_{\{x_1,\,x_2\}\subset\gamma}
\int_{(\mathbb R^d)^2}m(dh_1\times dh_2)q(x_1,x_2,h_1,h_2)\\
&\qquad\times \big(F(\gamma\setminus\{x_1,x_2\}\cup\{
x_1+h_1,x_2+h_2\})-F(\gamma)\big)G(\gamma\setminus\{x_1,x_2\}\cup\{
x_1+h_1,x_2+h_2\})\\
&\quad =\frac12\int_\Gamma\pi_z(d\gamma)\int_{\R^d}z\, dx_1\int_{\R^d} z\,dx_2
\int_{(\mathbb R^d)^2}m(dh_1\times dh_2)q(x_1,x_2,h_1,h_2)\\
&\qquad\times \big( F(\gamma\cup\{x_1+h_1,x_2+h_2\})-F(\gamma\cup\{x_1,x_2\})\big)
G(\gamma\cup\{x_1+h_1,x_2+h_2\})\\
&\quad =\frac12\int_\Gamma\pi_z(d\gamma)\int_{(\mathbb R^d)^2}m(dh_1\times dh_2)\int_{\R^d}z\, dx_1\int_{\R^d} z\,dx_2\,
q(x_1-h_1,x_2-h_2,h_1,h_2)\\
&\qquad\times \big( F(\gamma\cup\{x_1,x_2\})-F(\gamma\cup\{x_1-h_1,x_2-h_2\})\big)
G(\gamma\cup\{x_1,x_2\})\\
&\quad =\frac12\int_\Gamma\pi_z(d\gamma)\int_{\R^d}z\, dx_1\int_{\R^d} z\,dx_2
\int_{(\mathbb R^d)^2}m'(dh_1\times dh_2)q(x_1+h_1,x_2+h_2,-h_1,-h_2)\\
&\qquad\times \big( F(\gamma\cup\{x_1,x_2\})-F(\gamma\cup\{x_1+h_1,x_2+h_2\})\big)
G(\gamma\cup\{x_1,x_2\})\\
&\quad =\frac12\int_\Gamma\pi_z(d\gamma)\int_{\R^d}z\, dx_1\int_{\R^d} z\,dx_2
\int_{(\mathbb R^d)^2}m(dh_1\times dh_2)q(x_1,x_2,h_1,h_2)\\
&\qquad\times \big( F(\gamma\cup\{x_1,x_2\})-F(\gamma\cup\{x_1+h_1,x_2+h_2\})\big)
G(\gamma\cup\{x_1,x_2\})\\
&\quad=-\int_\Gamma\pi_z(d\gamma)\sum_{\{x_1,\,x_2\}\subset\gamma}
\int_{(\mathbb R^d)^2}m(dh_1\times dh_2)q(x_1,x_2,h_1,h_2)\\
&\qquad\times \big(F(\gamma\setminus\{x_1,x_2\}\cup\{
x_1+h_1,x_2+h_2\})-F(\gamma)\big)G(\gamma).
\end{align*}
From here \eqref{hcfthg} follows.
\end{proof}

As easily seen, $(\mathcal E,\FC)$ is a pre-Dirichlet form, i.e., if this form is closable in $L^2(\Gamma,\pi_z)$, then it is a Dirichlet form, see e.g.\ \cite{Fu80,MR} for details on Dirichlet forms.

\begin{lemma}\label{fdrytrs}
Assume that the following two conditions are satisfied:

\rom{(C1)} For each  $\Lambda\in\mathcal B_0(\R^d)$
$$ \int_{(\R^d)^2}Q(x_1,x_2,dh_1\times dh_2)(\mathbf 1_\Lambda(x_1)+\mathbf 1_\Lambda(x_1+h_1))\in L^1((\R^d)^2,dx_1\,dx_2)\cap L^2((\R^d)^2,dx_1\,dx_2). $$

\rom{(C2)} We have
$$\sup_{x_1\in\mathbb R^d}\int_{\R^d} dx_2\int_{(\R^d)^2}Q(x_1,x_2,dh_1\times dh_2)<\infty.$$
Then, for each $F\in \FC$, $LF\in L^2(\Gamma,\pi_z)$, and so $(-L,\FC)$ is the generator of the quadratic form $(\mathcal E,\FC)$ on $L^2(\Gamma,\pi_z)$.
\end{lemma}

\begin{proof}
Using the Mecke identity \eqref{Mecke}, we easily derive the following formula:
\begin{multline}
\int_\Gamma\pi_z(d\gamma)\left(\sum_{\{x_1,\,x_2\}\subset\gamma}f(x_1,x_2)\right)^2\\
=\frac14\left(\int_{\R^d}z\,dx_1\int_{\R^d}z\,dx_2\, f(x_1,x_2)\right)^2+\int_{\R^d}z\,dx_1\int_{\R^d}z\,dx_2\int_{\R^d}z\,dx_3\, f(x_1,x_2)f(x_2,x_3)\\
\text{}+\frac12\int_{\R^d}z\,dx_1\int_{\R^d}z\,dx_2\,f(x_1,x_2)^2
\label{fdrssgdtd}\end{multline}
for any measurable function $f:(\mathbb R^d)^2\to[0,\infty]$ satisfying $f(x_1,x_2)=f(x_2,x_1)$ for all $x_1,x_2\in\R^d$. For any $F\in\FC$, there exists a  $\Lambda\in\mathcal B_0(\R^d)$ such that
 \begin{multline}\big|F(\gamma\setminus\{x_1,x_2\}\cup\{x_1+h_1,x_2+h_2\})-F(\gamma)\big|\\
 \le C_1(\mathbf 1_\Lambda(x_1)+\mathbf 1_\Lambda(x_2)+\mathbf 1_\Lambda(x_1+h_1)+\mathbf 1_\Lambda(x_2+h_2)).\label{dse5sw}\end{multline}
 Here and below we denote by $C_i$, $i=1,2,3,\dots$, strictly positive constants whose explicit value is not important for us.
 Now the statement of the lemma follows from (C1), (C2), \eqref{dtsrea}, \eqref{fdrssgdtd}, and \eqref{dse5sw}.
\end{proof}

Completely analogously to the proof of Theorem~3.1 in \cite{KLR} (see also the proof of Theorem~3.1 in \cite{LO}),
we easily conclude the following theorem from  Lemmas \ref{vtrsdty} and~\ref{fdrytrs}.

\begin{theorem}\label{e6ue6} Assume that  conditions \eqref{dtsrea}, \eqref{cfsres}, \rom{(C1)}, and \rom{(C2)} are satisfied.
Then the quadratic  form $(\mathcal{E},
\FC)$ is closable in $L^{2}(\Gamma,
\pi_z)$ and its closure will be denoted by $(\mathcal{E},
D(\mathcal{E}))$.
Further there exists a conservative Hunt process
$$
M=\left(\Omega,
\mathcal{F},(\mathcal{F}_{t})_{t\geq
0},(\Theta_{t})_{t\geq 0}, (X(t))_{t\geq 0},
(P_{\gamma})_{\gamma\in \Gamma}\right)
$$
on $\Gamma$ which is properly associated with
$(\mathcal{E}, D(\mathcal{E}))$, i.e., for each
($\pi_z$-version of) $F\in L^{2}(\Gamma, \pi_z)$ and $t>0$ $$\Gamma\ni
\gamma\mapsto
(p_{t}F)(\gamma):=\int_{\Omega}F(X(t))\, dP_{\gamma}$$
is an $\mathcal{E}$-quasi continuous version of
$\exp(tL)F$. Here  $(-L, D(L))$ is the
generator of the quadratic form $(\mathcal{E}, D(\mathcal{E}))$---the Friedrichs extension of the operator $(-L,\FC)$.
$M$ is up-to $\pi_z$-equivalence unique. In particular,
$M$ is $\pi_z$-symmetric and has $\pi_z$ as invariant measure.
\end{theorem}

\begin{remark} We refer to \cite{MR} for an explanation of notations appearing in Theorem~\ref{e6ue6}, see also a brief explanation of them in \cite{LO}.
\end{remark}

We will call a Markov process as  in Theorem~\ref{e6ue6} a  stochastic dynamics of binary jumps.  Let us now consider two classes of such  dynamics. \vspace{2mm}

 {\bf 1)}  Let us assume that the  measure $m(dh_1\times dh_2)$ in \eqref{xsreasr} is the Lebesgue measure $dh_1\, dh_2$, and let us assume that $ q(x_1,x_2,h_1,h_2)=q(x_2-x_1,h_1,h_2)$
for some measurable function $q:(\R^d)^3\to[0,\infty]$. (Here and below we are using an obvious abuse of notation.)
Thus,
\begin{equation}\label{tsres}
Q(x_1,x_2,dh_1\times dh_2)=dh_1\,dh_2\, q(x_2-x_1,h_1,h_2).
\end{equation}

\begin{proposition}\label{fddtr} Assume that \eqref{tsres} holds and
\begin{align}
q(-x,h_1,h_2)&=q(x,h_2,h_1),\label{dtrdss}\\
q(x,h_1,h_2)&=q(x+h_2-h_1,-h_1,-h_2)\label{dseraw}
\end{align}
for all $x,h_1,h_2\in\R^d$.
 Further assume that
\begin{gather}
q(x,h_1,h_2)\in L^1((\mathbb R^d)^3,dx\,dh_1\,dh_2),\label{vftsers}\\
\operatornamewithlimits{ess\,sup}_{x\in\R^d}\int_{\R^d}dh_1\int_{\R^d}dh_2\,q(x,h_1,h_2)<\infty.\label{fdt}
\end{gather}
Then $Q(x_1,x_2,dh_1\times dh_2)$ satisfies the conditions of Theorem \rom{\ref{e6ue6}}.
\end{proposition}

\begin{proof}
By \eqref{tsres}, conditions \eqref{dtsrea}, \eqref{cfsres}  reduce to \eqref{dtrdss}, \eqref{dseraw}. Condition
\eqref{vftsers} clearly implies (C2), so we only have to check (C1).
For each  $\Lambda\in\mathcal B_0(\R^d)$, \eqref{vftsers} implies that
$$\int_{(\R^d)^2}Q(x_1,x_2,dh_1\times dh_2)\mathbf 1_\Lambda(x_1)\in L^1((\R^d)^2,dx_1\,dx_2).$$
By  \eqref{dseraw} and  \eqref{vftsers},
\begin{align*}
&\int_{\R^d}dx_1 \int_{\R^d}dx_2 \int_{(\R^d)^2}Q(x_1,x_2,dh_1\times dh_2)\mathbf 1_\Lambda(x_1+h_1)\\
&\quad =\int_{\R^d}dx_1 \int_{\R^d}dx_2 \int_{\R^d}dh_1 \int_{\R^d}dh_2\, q(x_2-x_1+h_2-h_1,-h_1,-h_2)\mathbf 1_\Lambda(x_1+h_1)\\
&\quad =\int_{\Lambda}dx_1 \int_{\R^d}dx_2 \int_{\R^d}dh_1 \int_{\R^d}dh_2\, q(x_2-x_1,-h_1,-h_2)<\infty.
\end{align*}
Analogously, using additionally \eqref{fdt}, we get
\begin{align*}
&\int_{\R^d}dx_1 \int_{\R^d}dx_2 \left(\int_{(\R^d)^2}Q(x_1,x_2,dh_1\times dh_2)\mathbf 1_\Lambda(x_1)\right)^2\\
&\quad= \int_{\Lambda}dx_1 \int_{\R^d}dx_2\int_{\R^d}dh_1 \int_{\R^d}dh_2\, q(x_2-x_1,h_1,h_2) \int_{\R^d}dh'_1 \int_{\R^d}dh'_2\,  q(x_2-x_1,h'_1,h'_2)\\
&\quad\le C_2 \int_{\Lambda}dx_1 \int_{\R^d}dx_2\int_{\R^d}dh_1 \int_{\R^d}dh_2\, q(x_2-x_1,h_1,h_2) <\infty,
\end{align*}
and
\begin{align*}
&\int_{\R^d}dx_1 \int_{\R^d}dx_2 \left(\int_{(\R^d)^2}Q(x_1,x_2,dh_1\times dh_2)\mathbf 1_\Lambda(x_1+h_1)\right)^2\\
&\quad\le \int_{\R^d}dx_1 \int_{\R^d}dx_2\int_{\R^d}dh_1 \int_{\R^d}dh_2\, q(x_2-x_1,h_1,h_2)\mathbf 1_\Lambda(x_1+h_1)\\
&\qquad\times \int_{\R^d}dh'_1 \int_{\R^d}dh'_2\,  q(x_2-x_1,h'_1,h'_2)\\
&\quad\le C_2\int_{\R^d}dx_1 \int_{\R^d}dx_2\int_{\R^d}dh_1 \int_{\R^d}dh_2\, q(x_2-x_1,h_1,h_2)\mathbf 1_\Lambda(x_1+h_1)<\infty.
\end{align*}
Thus, (C1) is satisfied. \end{proof}

The following proposition, whose proof is straightforward, presents a possible choice of a function $q(x,h_1,h_2)$.

\begin{proposition}\label{dersres}
Assume that functions $a:\R^d\to [0,\infty]$ and $b:\R^d\to[0,\infty]$ are measurable and even, i.e.,
$a(-x)=a(x)$, $b(-x)=b(x)$, $x\in\R^d$. Further assume that
$$ a,b\in L^1(\R^d,dx),\quad \operatornamewithlimits{ess\, sup}_{x\in\R^d}b(x)<\infty.$$
Set
$$ q(x,h_1,h_2):=a(h_1)a(h_2)\big(b(x)+b(x+h_2-h_1)\big).$$
Then the function $q(x,h_1,h_2)$ satisfies the conditions of Proposition \rom{\ref{fddtr}}, and so \linebreak $Q(x_1,x_2,dh_1\times dh_2)$ given by \eqref{tsres} satisfies the conditions of Theorem \rom{\ref{e6ue6}}.
\end{proposition}\vspace{2mm}

{\bf 2)} The following class of  stochastic dynamics of binary jumps is inspired by the paper \cite{BMT}. Let us assume that  measure $m(dh_1\times dh_2)$ is the pushforward of the Lebesgue measure $dh$ on $\R^d$ under the mapping $\R^d\ni h\mapsto(h,-h)\in(\R^d)^2$. Let us further assume that $ q(x_1,x_2,h_1,h_2)=q(x_2-x_1,h_1)$
for some measurable function $q:(\R^d)^2\to[0,\infty]$. Thus, for any measurable function $f:(\R^d)^4\to[0,\infty]$,
\begin{equation}\label{sfzaewa}
\int_{(\R^d)^2}Q(x_1,x_2,dh_1\times dh_2)f(x_1,x_2,h_1,h_2)=\int_{\R^d}dh\,q(x_2-x_1,h)f(x_1,x_2,h,-h).
\end{equation}
Hence, for any $F\in\FC$,
$$ (LF)(\gamma)=\sum_{\{x_1,\,x_2\}\subset\gamma}\int_{\R^d}dh\, q(x_2-x_1,h)
\big(F(\gamma\setminus\{x_1,x_2\}\cup\{x_1+h,x_2-h\})-F(\gamma)\big).$$
Completely analogously to Propositions \ref{fddtr}, \ref{dersres}, we get

\begin{proposition}\label{cdsesers} \rom{i)} Assume that \eqref{sfzaewa} holds and
\begin{align*}
q(-x,h)&=q(x,-h),\\
q(x,h)&=q(-x+2h,h)
\end{align*}
for  $x,h\in\R^d$. Further assume that
\begin{gather*}
q(x,h)\in L^1((\R^d)^2,dx\,dh),\\
\operatornamewithlimits{ess\,sup}_{x\in\R^d}\int_{\R^d}dh\,q(x,h)<\infty.
\end{gather*}
Then $Q(x_1,x_2,dh_1\times dh_2)$ satisfies the conditions of Theorem \rom{\ref{e6ue6}}.

\rom{ii)} Let functions $a$, $b$ be as in Proposition \rom{\ref{dersres}}. Set
$$ q(x,h):=a(h)b(x-h).$$
Then the function $q(x,h)$ satisfies the conditions of\/ \rom{i)}, and so  $Q(x_1,x_2,dh_1\times dh_2)$ given by \eqref{sfzaewa} satisfies the conditions of Theorem \rom{\ref{e6ue6}}.
\end{proposition}\vspace{2mm}

\section{Uniqueness of dynamics}\label{gufr75eses}

We will now show that the Markov pre-generator \eqref{tyrsar5a} defined on $\FC$, with $Q(x_1,x_2,dh_1\times dh_2)$ being as in Proposition~\ref{fddtr}, or as in Proposition~\ref{cdsesers}, i) uniquely identifies a Markov process on $\Gamma$.

\begin{theorem}\label{huigfmkih}
Let $Q(x_1,x_2,dh_1\times dh_2)$ be either as in Proposition~\ref{fddtr}, or as in Proposition~\ref{cdsesers}, i). Then the operator $(-L,\FC)$  is essentially selfadjoint in $L^2(\Gamma,\pi_z)$, so that $(-L,D(L))$ is the closure  of $(-L,\FC)$ in $L^2(\Gamma,\pi_z)$.
\end{theorem}

\begin{proof} We will only prove the theorem in the case of the dynamics as in Proposition~\ref{fddtr}, which is the harder case.
Denote by $(-\overline{L},D(\overline{L}))$ the closure of the symmetric operator $(-L,\FC)$ in $L^2(\Gamma,\pi_z)$. Thus, the operator $(-L,D(L))$ is an extension of $(-\overline{L},D(\overline{L})$.
Analogously to the proof of Lemma~\ref{fdrytrs} and Proposition~\ref{fddtr}, one can show that, for each $f\in C_0(\R^d)$ and $n\in\mathbb N$, $\la f,\cdot\ra^n\in D(\overline{L})$. Hence, by the polarization identity (e.g.\ \cite[Chap.~2, formula~(2.7)]{BK}),  we have
\begin{equation}\label{dsrer}\la f_1,\cdot\ra\dotsm \la f_n,\cdot\ra\in D(\overline{L}), \quad f_1,\dots,f_n\in C_0(\R^d),\ n\in \mathbb N.\end{equation}
Let $\mathscr P$ denote the set of all functions on $\Gamma$ which are finite sums of functions as in \eqref{dsrer} and constants. Thus, $\mathscr P$ is a set of polynomials on $\Gamma$, and $\mathscr P\subset D(\overline L)$.

For a real Hilbert space $\mathcal H$, denote by $\mathcal F(\mathcal H)$ the symmetric Fock space over $\mathcal H$.
Thus, $\mathcal F(\mathcal H)$ is the Hilbert space
$$ \mathcal F(\mathcal H)=\bigoplus_{n=0}^\infty \mathcal F^{(n)}(\mathcal H),$$
where $\mathcal F^{(0)}(\mathcal H):=\mathbb R$, and for $n\in\mathbb N$,  $\mathcal F^{(n)}(\mathcal H)$ coincides with $\mathcal H^{\odot n}$ as a set, and for any $f^{(n)},g^{(n)}\in \mathcal F^{(n)}(\mathcal H)$
$$ (f^{(n)},g^{(n)})_{\mathcal F(\mathcal H)}:=(f^{(n)},g^{(n)})_{\mathcal H^{\odot n}}n!\,.$$
 Here $\odot$ stands for symmetric tensor product.

Let
$$ I: L^2(\Gamma,\pi_z)\to\mathcal F(L^2(\R^d,z\,dx))$$
denote the unitary isomorphism which is derived through multiple stochastic integrals with respect to the centered Poisson random measure on $\R^d$ with intensity measure $z\,dx$, see e.g.\ \cite{Surgailis}. Denote by $\widetilde{\mathscr P}$ the subset of $\mathcal F(L^2(\R^d,z\,dx))$ which is the linear span of vectors of the form $$f_1\odot f_2\odot\dots\odot f_n,\quad f_1,\dots,f_n\in C_0(\R^d),\ n\in \mathbb N,$$
and the vacuum vector $\Psi=(1,0,0,\dots)$. For any $f\in C_0(\R^d)$, denote by $M_f$ the operator of multiplication by the function $\la f,\cdot\ra$ in $L^2(\Gamma,\pi_z)$. Using the representation of the operator $IM_fI^{-1}$ as a sum of creation, neutral, and annihilation operators in the Fock space (see e.g.\ \cite{Surgailis}), we easily conclude that $I \mathscr P=\widetilde{\mathscr P}$.

We define a quadratic form $(\widetilde{\mathcal E},D(\widetilde{\mathcal E}))$ by $$\widetilde{\mathcal E}(f,g):={\mathcal E}(I^{-1}f,I^{-1}g),\quad f,g\in D(\widetilde{\mathcal E}):=ID(\mathcal E),$$
on $\mathcal F(L^2(\R^d,z\,dx))$, and let $(-\widetilde L, D(\widetilde L))$ denote the generator of this form. Thus, $D(\widetilde L)=ID(L)$ and $\widetilde L= ILI^{-1}$ on $D(\widetilde L)$.

For each $x\in\R^d$, we define an annihilation operator at $x$ as follows: $\di_x: \widetilde{\mathscr P}\to \widetilde{\mathscr P}$ is a linear mapping given through
$$ \di_x\Psi:=0,\quad \di_x f_1\odot f_2\odot\dots\odot f_n:=\sum_{i=1}^n f_i(x)\, f_1\odot f_2\odot\dots\odot \check f_i \odot \dots\odot f_n,$$
where $\check f_i$ denotes the absence of  $f_i$. We will preserve the notation $\di_x$
for the operator $I\di_x I^{-1}:\mathscr P\to\mathscr P$. This operator admits the following explicit representation
$$ \di_x F(\gamma)= F(\gamma\cup \{x\})-F(\gamma)$$
for $\pi_z$-a.a.\ $\gamma\in\Gamma$, see e.g.\ \cite{AKR,IK,NV}.
By  the Mecke formula, for any $F\in\mathscr P$,
\begin{align*}&\mathcal E(F,F)=\frac14\int_{\R^d}z\,dx_1\int_{\R^d}z\,dx_2 \,
\int_{\R^d}dh_1\int_{\R^d}dh_2\, q(x_2-x_1,h_1,h_2)\\
&\qquad\times \int_\Gamma \pi_z(d\gamma)
\big(F(\gamma\cup\{x_1+h_1,x_2+h_2\})-F(\gamma\cup\{x_1,x_2\})\big)^2.
  \end{align*}
  Noting that
  $$F(\gamma\cup\{x_1,x_2\})-F(\gamma)=(\di_{x_1}\di_{x_2}-\di_{x_1}-\di_{x_2})F(\gamma),$$
we thus get, for any $f\in\widetilde{\mathscr P}$,
\begin{align}
&\widetilde{\mathcal E}(f,f)=\frac14
\int_{\R^d}z\,dx_1\int_{\R^d}z\,dx_2 \,
\int_{\R^d}dh_1\int_{\R^d}dh_2\, q(x_2-x_1,h_1,h_2)\notag\\
&\qquad\times\big\|
(\di_{x_1+h_1}\di_{x_2+h_2}-\di_{x_1+h_1}-\di_{x_2+h_2}-\di_{x_1}\di_{x_2}+\di_{x_1}+\di_{x_2})f
\big\|^2_{\mathcal F(L^2(\R^d,z\,dx))}.
  \label{fdsyuio}\end{align}
Hence, at least heuristically, the generator of this form  has representation
\begin{align} -\widetilde L&=\frac14
\int_{\R^d}z\,dx_1\int_{\R^d}z\,dx_2 \,
\int_{\R^d}dh_1\int_{\R^d}dh_2\, q(x_2-x_1,h_1,h_2)\notag\\
&\quad\times(\di^\dag_{x_1+h_1}\di^\dag_{x_2+h_2}-\di^\dag_{x_1+h_1}-\di^\dag_{x_2+h_2}-\di^\dag_{x_1}\di^\dag_{x_2}+\di^\dag_{x_1}+\di^\dag_{x_2})\notag\\
&\quad\times(\di_{x_1+h_1}\di_{x_2+h_2}-\di_{x_1+h_1}-\di_{x_2+h_2}-\di_{x_1}\di_{x_2}+\di_{x_1}+\di_{x_2}),
\label{vctese}\end{align}
where   $\di^\dag_x$ denotes a creation operator at point $x\in\R^d$ ($\di^\dag_x$ being rather an operator-valued distribution, see e.g. \cite{HKPS})
Noting that the operators $\di_x$, $x\in\R^d$, commute, we get from \eqref{vctese} and \eqref{dtrdss}, \eqref{dseraw}:
$$-\widetilde L_0=J^++J^0+J^-,$$
where
\begin{align}
&J^+:=\int_{\R^d}z\,dx_1\int_{\R^d}z\,dx_2 \,
\int_{\R^d}dh_1\int_{\R^d}dh_2\, q(x_2-x_1,h_1,h_2)\big(-\di^\dag_{x_1}\di^\dag_{x_2}\di_{x_1}+\di^\dag_{x_1}\di^\dag_{x_2}\di_{x_1+h_1}\big),\notag\\
&J^0:= \int_{\R^d}z\,dx_1\int_{\R^d}z\,dx_2 \,
\int_{\R^d}dh_1\int_{\R^d}dh_2\, q(x_2-x_1,h_1,h_2)\bigg(\frac12\di^\dag_{x_1}\di^\dag_{x_2}\di_{x_1}\di_{x_2}\notag\\
&\qquad\text{}
+\frac12\di^\dag_{x_1}\di^\dag_{x_2}\di_{x_1+h_1}\di_{x_2+h_2}+\di^\dag_{x_1}\di_{x_1}+\di^\dag_{x_1}\di_{x_2}-\di^\dag_{x_1}\di_{x_1+h_1}-\di^\dag_{x_1}\di_{x_2+h_2}\bigg),
\end{align}
and $J^-$ is the formal adjoint of $J^+$.
Note that the operators of creation and annihilation in the above formulas are in Wick order, i.e., the creation operators act after the annihilation operators. Thus, one may hope to give a rigorous sense to the above integrals by using the corresponding quadratic forms, see e.g.\cite[Chapter X.7]{RS}.

Denote by $\mathcal F_{\mathrm{fin}}(L^2(\R^d,z\,dx))$ the subset of $\mathcal F(L^2(\R^d,z\,dx))$ consisting of all finite sequences $f=(f^{(0)}, f^{(1)},\dots,f^{(n)},0,0,\dots)$, $f^{(i)}\in \mathcal F^{(i)}(L^2(\R^d,z\,dx))$, $i=1,\dots,n$, $n\in\mathbb N$. Clearly $\widetilde{\mathscr P}\subset \mathcal F_{\mathrm{fin}}(L^2(\R^d,z\,dx))$. Using \eqref{fdsyuio}, we conclude by approximation that $\mathcal F_{\mathrm{fin}}(L^2(\R^d,z\,dx))\subset D(\widetilde {\mathcal E})$.

Using the corresponding quadratic form, we get, for each
$f^{(n)}\in \mathcal F^{(n)}(L^2(\R^d,z\,dx))$,
\begin{align*}
&\bigg(\int_{\R^d}z\,dx_1\int_{\R^d}z\,dx_2 \,
\int_{\R^d}dh_1\int_{\R^d}dh_2\, q(x_2-x_1,h_1,h_2)\di^\dag_{x_1}\di^\dag_{x_2}\di_{x_1}
f^{(n)}\bigg)(y_1,\dots,y_{n+1})\\
&\quad=n\bigg(f^{(n)}(y_1,\dots,y_n)\int_{\R^d}dh_1\int_{\R^d}dh_2\, q(y_{n+1}-y_n,h_1,h_2)\bigg)^\sim,
\end{align*}
where $(\cdot)^\sim$ denotes symmetrization. We have, by \eqref{vftsers} and \eqref{fdt}:
\begin{align*}
&\int_{(\R^d)^{n+1}} \bigg(f^{(n)}(y_1,\dots,y_n)\int_{\R^d}dh_1\int_{\R^d}dh_2\, q(y_{n+1}-y_n,h_1,h_2)\bigg)^2 z\,dy_1\dotsm z\,dy_{n+1}\\
&\quad\le C_3\int_{(\R^d)^n}z\,dy_1\dotsm z\,dy_{n}\,
f^{(n)}(y_1,\dots,y_n)^2\int_{\R^d} z\,dy_{n+1}\int_{\R^d}dh_1\int_{\R^d}dh_2\, q(y_{n+1}-y_n,h_1,h_2)\\
&\quad\le C_4\|f^{(n)}\|^2_{L^2((\R^d)^n,z\,dy_1\dotsm z\,dy_n)}.
\end{align*}
Analogously,
\begin{align*}
&\bigg(\int_{\R^d}z\,dx_1\int_{\R^d}z\,dx_2
\int_{\R^d}dh_1\int_{\R^d}dh_2\, q(x_2-x_1,h_1,h_2)\di^\dag_{x_1}\di^\dag_{x_2}\di_{x_1+h_1}
f^{(n)}\bigg)(y_1,\dots,y_{n+1})\\
&\quad=n\bigg(\int_{\R^d}dh_1\int_{\R^d}dh_2\, q(y_{n+1}-y_n,h_1,h_2)f^{(n)}(y_1,\dots,y_{n-1},y_n+h_1)\bigg)^\sim,
\end{align*}
and by the Cauchy inequality
\begin{align*}
&\int_{(\R^d)^{n+1}} \bigg(\int_{\R^d}dh_1\int_{\R^d}dh_2\, q(y_{n+1}-y_n,h_1,h_2)f^{(n)}(y_1,\dots,y_{n-1},y_n+h_1)\bigg)^2 z\,dy_1\dotsm z\,dy_{n+1}\\
&\quad\le \int_{(\R^d)^{n+1}}z\,dy_1\dotsm z\,dy_{n+1}\int_{\R^d}dh_1\int_{\R^d}dh_2\int_{\R^d}dh'_1\int_{\R^d}dh'_2\notag \\
&\qquad\times q(y_{n+1}-y_n,h_1,h_2)q(y_{n+1}-y_n,h'_1,h'_2)f^{(n)}(y_1,\dots,y_{n-1},y_n+h_1)^2\notag\\
&\quad\le C_5\int_{(\R^d)^{n+1}}z\,dy_1\dotsm z\,dy_{n+1}\int_{\R^d}dh_1\int_{\R^d}dh_2\, q(y_{n+1}-y_n,h_1,h_2)\notag\\
&\qquad\times f^{(n)}(y_1,\dots,y_{n-1},y_n+h_1)^2\notag\\
&\quad =C_5 \int_{(\R^d)^{n}}z\,dy_1\dotsm z\,dy_{n}\, f^{(n)}(y_1,\dots,y_n)^2\int_{\R^d}dh_1\int_{\R^d}dh_2\int_{\R^d}z\,dy_{n+1}\notag\\
&\qquad\times q(y_{n+1}-y_n+h_1,h_1,h_2)\\
&\quad= C_6 \|f^{(n)}\|^2_{L^2((\R^d)^n,z\,dy_1\dotsm z\,dy_n)}.
\end{align*}
Hence, $J^+$ can be realized as a linear operator on $\mathcal F_{\mathrm{fin}}(L^2(\R^d,z\,dx))$ and
 \begin{equation}\label{hjcxhs}\|J^+\restriction \mathcal F^{(n)}(L^2(\R^d,z\,dx))\|_{\mathscr L(\mathcal F^{(n)}(L^2(\R^d,z\,dx)),\,\mathcal F^{(n+1)}(L^2(\R^d,z\,dx)))} \le C_7\, n\sqrt{n+1},\end{equation}
where the constant $C_7$ is independent of $n$.
Furthermore,  $J^-$ can be realized as the restriction to $\mathcal F_{\mathrm{fin}}(L^2(\R^d,z\,dx))$ of the adjoint operator of $J^+$, and by \eqref{hjcxhs}
 \begin{equation}\label{wvvf}\|J^-\restriction \mathcal F^{(n+1)}(L^2(\R^d,z\,dx))\|_{\mathscr L(\mathcal F^{(n+1)}(L^2(\R^d,z\,dx)),\,\mathcal F^{(n)}(L^2(\R^d,z\,dx)))} \le C_7\,n\sqrt{n+1}.\end{equation}

 Using again the corresponding quadratic form, we get
 \begin{align*}
 &(J_1^0 f^{(n)})(y_1,\dots,y_n):=\\
 &\quad=\bigg(
 \int_{\R^d}z\,dx_1\int_{\R^d}z\,dx_2
\int_{\R^d}dh_1\int_{\R^d}dh_2\, q(x_2-x_1,h_1,h_2)\di^\dag_{x_1}\di^\dag_{x_2}\di_{x_1}\di_{x_2} f^{(n)} \bigg)(y_1,\dots,y_n)\\
&\quad=n(n-1)\bigg(f^{(n)}(y_1,\dots,y_n)
\int_{\R^d}dh_1\int_{\R^d}dh_2\, q(y_2-y_{1},h_1,h_2)\bigg)^\sim,
\end{align*}
and hence
  \begin{equation}\label{jcfkugftcf}\|J^0_1\restriction \mathcal F^{(n)}(L^2(\R^d,z\,dx))\|_{\mathscr L(\mathcal F^{(n)}(L^2(\R^d,z\,dx)),\,\mathcal F^{(n)}(L^2(\R^d,z\,dx)))} \le
 C_8\, n(n-1)  .\end{equation}
 Analogously,
\begin{align*}
 &(J_2^0 f^{(n)})(y_1,\dots,y_n):=\\
 &\quad=\bigg(
 \int_{\R^d}z\,dx_1\int_{\R^d}z\,dx_2
\int_{\R^d}dh_1\int_{\R^d}dh_2\, q(x_2-x_1,h_1,h_2)\di^\dag_{x_1}\di^\dag_{x_2}\di_{x_1+h_1}\di_{x_2+h_2} f^{(n)} \bigg)(y_1,\dots,y_n)\\
&\quad=n(n-1)\bigg(\int_{\R^d}dh_1\int_{\R^d}dh_2\,
 q(y_2-y_{1},h_1,h_2)f^{(n)}(y_1+h_1,y_2+h_2,y_3\dots,y_n)\bigg)^\sim.
\end{align*}
 We have, by the Cauchy inequality,
\begin{align*}
&\int_{(\R^d)^n}\bigg(\int_{\R^d}dh_1\int_{\R^d}dh_2\,
 q(y_2-y_{1},h_1,h_2)f^{(n)}(y_1+h_1,y_2+h_2,y_3\dots,y_n)\bigg)^2z\,dy_1\dotsm z\,dy_n\\
 &\quad\le \int_{(\R^d)^n}z\,dy_1\dotsm z\,dy_n
 \int_{\R^d}dh_1\int_{\R^d}dh_2\int_{\R^d}dh'_1\int_{\R^d}dh'_2\\
 &\qquad\times q(y_2-y_{1},h_1,h_2)q(y_2-y_{1},h'_1,h'_2)f^{(n)}(y_1+h_1,y_2+h_2,y_3\dots,y_n)^2\\
 &\quad\le C_9 \int_{(\R^d)^n}z\,dy_1\dotsm z\,dy_n
 \int_{\R^d}dh_1\int_{\R^d}dh_2\, q(y_2-h_2-y_1+h_1,h_1,h_2)f^{(n)}(y_1,\dots,y_n)^2\\
 &\quad= C_9 \int_{(\R^d)^n}z\,dy_1\dotsm z\,dy_n
 \int_{\R^d}dh_1\int_{\R^d}dh_2\, q(y_2+h_2-y_1-h_1,-h_1,-h_2)f^{(n)}(y_1,\dots,y_n)^2\\
 &\quad = C_9 \int_{(\R^d)^n}z\,dy_1\dotsm z\,dy_n
 \int_{\R^d}dh_1\int_{\R^d}dh_2\, q(y_2-y_1,h_1,h_2)f^{(n)}(y_1,\dots,y_n)^2\\
 &\quad\le C_{10}\|f^{(n)}\|^2_{L^2((\R^d)^n,z\,dy_1\dotsm z\,dy_n)}.
 \end{align*}
 Therefore, an estimate similar to \eqref{jcfkugftcf} holds for $J_2^0$.

 We next have:
 \begin{align*}
 &(J_3^0 f^{(n)})(y_1,\dots,y_n):=\\
 &\quad=\bigg(
 \int_{\R^d}z\,dx_1\int_{\R^d}z\,dx_2
\int_{\R^d}dh_1\int_{\R^d}dh_2\, q(x_2-x_1,h_1,h_2)\di^\dag_{x_1}\di_{x_1} f^{(n)} \bigg)(y_1,\dots,y_n)\\
&\quad=n \left(\int_{\R^d}z\,dy\int_{\R^d}dh_1\int_{\R^d}dh_2\, q(y,h_1,h_2)\right) f^{(n)}(y_1,\dots,y_n).
\end{align*}
 Hence
 \begin{equation}\label{jhgtydsa5re}\|J^0_3\restriction \mathcal F^{(n)}(L^2(\R^d,z\,dx))\|_{\mathscr L(\mathcal F^{(n)}(L^2(\R^d,z\,dx)),\,\mathcal F^{(n)}(L^2(\R^d,z\,dx)))} \le
 C_{11}\, n  .\end{equation}
(In fact, in this case we have equality, rather than inequality.) Next
 \begin{align*}
 &(J_4^0 f^{(n)})(y_1,\dots,y_n):=\\
 &\quad=\bigg(
 \int_{\R^d}z\,dx_1\int_{\R^d}z\,dx_2
\int_{\R^d}dh_1\int_{\R^d}dh_2\, q(x_2-x_1,h_1,h_2)\di^\dag_{x_1}\di_{x_2} f^{(n)} \bigg)(y_1,\dots,y_n)\\
&\quad=n \bigg(\int_{\R^d}z\,dx\int_{\R^d}dh_1\int_{\R^d}dh_2\, q(x-y_1,h_1,h_2)f^{(n)}(x,y_2,\dots,y_n)\bigg)^\sim.
\end{align*}
Hence, by the Cauchy inequality, we easily conclude that $J_4^0$ satisfies an estimate similar to \eqref{jhgtydsa5re}. The two remaining terms  with $\di^\dag_{x_1}\di_{x_1+h_1}$ and $\di^\dag_{x_1}\di_{x_2+h_2}$ can be treated similarly. Thus, $J^0$ can be realized as a linear operator on $\mathcal F_{\mathrm{fin}}(L^2(\R^d,z\,dx))$ and
\begin{equation}\label{hgydsers}\|J^0\restriction \mathcal F^{(n)}(L^2(\R^d,z\,dx))\|_{\mathscr L(\mathcal F^{(n)}(L^2(\R^d,z\,dx)),\,\mathcal F^{(n)}(L^2(\R^d,z\,dx)))} \le
 C_{12}\, n(n-1)  .\end{equation}

 Denote by $(-\mathcal L,D(\mathcal L))$ the closure of the operator $(-\widetilde L,\widetilde {\mathscr P})$ in $\mathcal F(L^2(\R^d,z\,dx))$. Thus,
 $(-\widetilde L,D(\widetilde L))$ is an extension of
 $(-\mathcal L,D(\mathcal L))$.
 We now easily see that $$\mathcal F_{\mathrm{fin}}(L^2(\R^d,z\,dx))\subset D(\mathcal L)$$ and the action of $-\mathcal L$ on $\mathcal F_{\mathrm{fin}}(L^2(\R^d,z\,dx))$ is indeed given by the above formulas.
  By \eqref{hjcxhs}, \eqref{wvvf} and \eqref{hgydsers}, for each $f^{(n)}\in \mathcal F^{(n)}(L^2(\R^d,z\,dx))$, there exists $t>0$ such that
$$\sum_{k=1}^\infty\frac{t^k}{(2k)!}\|(-\mathcal  L)^k f^{(n)}\|_{\mathcal F(L^2(\R^d,z\,dx))}<\infty.$$
Since $\mathscr P\subset D(\overline{L})$, we therefore conclude that  $I^{-1} \mathcal F_{\mathrm{fin}}(L^2(\R^d,z\,dx))\subset D( \overline{L})$ and for each $F\in I^{-1} \mathcal F_{\mathrm{fin}}(L^2(\R^d,z\,dx))$ there exists $t>0$ such that
\begin{equation}\label{kjdsrt}\sum_{k=1}^\infty\frac{t^k}{(2k)!}\|(- \overline L)^k F\|_{L^2(\Gamma,\pi_z)}<\infty.\end{equation}
Hence, by
 the Nussbaum theorem (see e.g.\ \cite[Theorem X.40]{RS}),
 the operator $(-\overline L,D(\overline L))$ is essentially selfadjoint on $I^{-1} \mathcal F_{\mathrm{fin}}(L^2(\R^d,z\,dx))$.
 From here the statement of the theorem follows.
\end{proof}

\section{Diffusion approximation}\label{vcydryds}
We will now consider a diffusion approximation for the stochastic dynamics as in Proposition~\ref{dersres}.

Denote by $\FCC$ the space of all functions of the form \eqref{drtses}, where  $N\in\N$, $\varphi_1,\dots,\varphi_N\in C^\infty_0(\mathbb R^d)$ and $g_F\in C^\infty_{\mathrm b}(\R^N)$. Here,
where $C^\infty_0(\mathbb R^d)$ and $C^\infty_{\mathrm b}(\R^N)$ denote the space of all smooth functions on $\R^d$ with compact support and the space of all smooth, bounded functions on $\R^N$ whose all derivatives are bounded, respectively. Clearly $\FCC\subset\FC$ and $\FCC$ is a dense subset of $L^2(\Gamma,\pi_z)$.

For a function $F\in \FCC$, $\gamma\in\Gamma$, and $x\in\gamma$, we denote
\begin{equation}\label{fdrtdtr}\nabla_x F(\gamma):= \nabla_y F(\gamma\setminus\{x\}\cup\{y\})\Big|_{y=x},\end{equation}
where $\nabla_y$ stands for the gradient in the $y$ variable. Analogously, we define a Laplacian
$\Delta_xF(\gamma)$.

We now scale the dynamics as follows. For each $\eps>0$,
we denote
\begin{equation}\label{kdstswr} q_\eps(x,h_1,h_2):=\eps^{-2d-2}a(h_1/\eps)a(h_2/\eps)\big(b(x)+b(x+h_2-h_1)\big),\end{equation}
and let $L_\eps$ denote the corresponding  $L$ operator.


\begin{theorem}\label{ydsets} Assume that
 functions $a:\R^d\to [0,\infty]$ and $b:\R^d\to[0,\infty)$ satisfy the following conditions:

\rom{a)} $a(x)=\tilde a(|x|)$, where $\tilde a:[0,\infty)\to[0,\infty]$ is a measurable function, and  $b$ is an even function;

\rom{b)} $a,b\in L^1(\mathbb R^d,dx)$;

\rom{c)} The function $a$ has a compact support;

\rom{d)} $b\in C_{\mathrm{b}}^1(\R^d)$, where $C_{\mathrm{b}}^1(\R^d)$ denotes the space of all bounded, continuously differentiable functions on $\R^d$ whose gradient is bounded;

\rom{e)} There exists $R>0$ such that
$$\int_{\R^d}dx\,\sup_{y\in B(x,R)}|\nabla b(y)|<\infty.$$
Here $B(x,R)$ denotes the closed ball in $\R^d$ centered at $x$ and of radius $R$.

Then, for each $F\in\FCC$,
$$ \text{$L_\eps F\to L_0F$ in $L^2(\Gamma,\pi_z)$  as $\eps\to0$.}$$
Here
\begin{equation}\label{jkfde6u} (L_0F)(\gamma):=c \sum_{x\in\gamma}\left[\Delta_x F(\gamma)
\sum_{y\in\gamma\setminus \{x\}}b(x-y)+\left\la\nabla_x F(\gamma),\sum_{y\in\gamma\setminus\{x\}}\nabla b(x-y)\right\ra\right],\end{equation}
where
\begin{equation}\label{trse} c:=\int_{\R^d}a(h)(h^1)^2\,dh,\end{equation}
$h^i$ denoting the $i$-th coordinate of $h\in\R^d$, $i=1,\dots,d$, and $\la\cdot,\cdot\ra$ stands for the scalar product in $\R^d$.
\end{theorem}

\begin{remark}
As will be seen from the proof of Theorem \ref{ydsets}, all the series on the right hand side of formula \eqref{jkfde6u} converge absolutely for $\pi_z$-a.a.\ $\gamma\in\Gamma$.
\end{remark}

\begin{remark}
For each $\gamma\in\Gamma$ and $x\in\gamma$, denote $A_x(\gamma):=c\sum_{y\in\gamma\setminus\{x\}}b(x-y)$. Also let $T\Gamma_\gamma:=L^2(\R^d\to\R^d,d\gamma)$ be a tangent space to $\Gamma$ at $\gamma$ (cf.\ \cite{AKR}). Then $\nabla F(\gamma)=(\nabla_x F(\gamma))_{x\in\gamma}\in T\Gamma_\gamma$ and set $B(\gamma):=\left(\sum_{y\in\gamma\setminus\{x\}}c\nabla b(x-y)\right)_{x\in\gamma}$. Then formula \eqref{jkfde6u} can be written in the form
$$ (L_0F)(\gamma)=\sum_{x\in\gamma}A_x(\gamma)\Delta_x F(\gamma)+\big\langle \nabla F(\gamma), B(\gamma)\big\rangle_{T\Gamma_\gamma}\,.$$
\end{remark}

\begin{remark} Note that condition e) of Theorem \ref{ydsets} is slightly stronger than the condition $|\nabla b|\in L^1(\R^d,dx)$.
\end{remark}

\begin{remark}
Recall that the generator of the gradient stochastic dynamics has the form
$$ (L'_0F)(\gamma):=\frac12 \sum_{x\in\gamma}\left[\Delta_x F(\gamma)
-\beta\left\la\nabla_x F(\gamma),\sum_{y\in\gamma\setminus\{x\}}\nabla \phi(x-y)\right\ra\right],$$ where $\beta$ is the inverse temperature and $\phi$ is the potential of pair interaction, see   \cite{AKR,Osada,Spohn,Yoshida} for further details. The difference between the generators $L_0$ and $L_0'$ is in the non-trivial  coefficient $A_x(\gamma)$ by $\Delta_xF(\gamma)$ in the operator $L_0$. This coefficient allows $L_0$ to be symmetric with respect to the Poisson measure, while $L_0'$ is symmetric with respect to the Gibbs measure corresponding to the inverse temperature $\beta$ and the potential of pair interaction $\phi$.
\end{remark}

\begin{proof}[Proof of Theorem \ref{ydsets}] Denote by $\ddot\Gamma$ the space of all multiple configurations in $\R^d$, i.e., $\ddot\Gamma$ consists of all $\{0,1,2,3,\dots,\infty\}$-valued Radon measures on $(\R^d,\mathcal B(\R^d))$, this space being also equipped with the vague topology. Evidently $\Gamma\subset\ddot\Gamma$. For any  $\gamma\in\ddot\Gamma$ and $f\in C_0(\R^d)$, we set $\la f,\gamma\ra:=\int_{\R^d}f(x)\,\gamma(dx)$. Hence, we can extend each function $F\in \FCC$ to $\ddot\Gamma$ by using  formula \eqref{drtses}.
(As easily seen, such an extension does not depend on the choice of the function's representation in the form \eqref{drtses}.)
  For each $\gamma\in\Gamma$ and  $x\in\gamma$, the function
$$\R^d\ni y\mapsto u(y):=F(\gamma-\delta_x+\delta_y)\in\R$$
is clearly smooth. Note that, while $\gamma-\delta_x\in\Gamma$, the measure $\gamma-\delta_x+\delta_y$ belongs to $\ddot\Gamma$ and not necessarily to $\Gamma$. For a fixed $y\in\R^d$, denote $\tilde\gamma:=\gamma-\delta_x+\delta_y$ and set $\nabla_x F(\tilde\gamma):=\nabla u(y)$. In the case where $y=x$ and so $\tilde\gamma=\gamma$, the just given definition of $\nabla_x F(\gamma)$ coincides with \eqref{fdrtdtr}.

By \eqref{tsres} and \eqref{kdstswr}, for $F\in \FCC$,
\begin{align}
&(L_\eps F)(\gamma)\notag\\
&=\sum_{\{x_1,\,x_2\}\subset\gamma}\eps^{-2}\int_{\R^d}dh_1\int_{\R^d}dh_2
\,a(h_1)a(h_2)\big(b(x_2-x_1)+b(x_2-x_1+\eps(h_2-h_1))\big)\notag\\
&\qquad\times \big(F(\gamma\setminus\{x_1,x_2\}\cup\{x_1+\eps h_1,x_2+\eps h_2\})-F(\gamma)\big)\notag\\
&=\sum_{\{x_1,\,x_2\}\subset\gamma}\eps^{-2}\int_{\R^d}dh_1\int_{\R^d}dh_2
\,a(h_1)a(h_2)\big(b(x_2-x_1)+b(x_2-x_1+\eps(h_2-h_1))\big)\notag\\
&\qquad\times \big(F(\gamma-\delta_{x_1}-\delta_{x_2}+\delta_{x_1+\eps h_1}+\delta_{x_2+\eps h_2})-F(\gamma)\big).\label{d5w34}
\end{align}
We have used the fact that, for any $\{x_1,x_2\}\subset\gamma$ and a.a.\ $(h_1,h_2)\in(\R^d)^2$, we have  $$\{x_1+\eps h_1,x_2+\eps h_2\}\cap\gamma=\varnothing.$$

For any $x,y\in\R^N$, $x\ne y$, $N\in\mathbb N$, denote by $[x,y]$ the line segment connecting points $x$ and $y$. By \eqref{d5w34}, condition d), and Taylor's formula, we get
\begin{multline}
(L_\eps F)(\gamma)=\sum_{\{x_1,\,x_2\}\subset\gamma}\eps^{-2}\int_{\R^d}dh_1\int_{\R^d}dh_2
\,a(h_1)a(h_2)\\
\times\big[
2b(x_2-x_1)+\la \nabla b(\tilde y(x_2-x_1,\eps(h_2-h_1))),\eps(h_2-h_1)\ra
\big]\\
\times\big[
\la \nabla_{(x_1,x_2)}F(\gamma),(\eps h_1,\eps h_2)\ra+(1/2)\la \nabla_{(x_1,x_2)}^2F(\tilde\gamma(\gamma,x_1,x_2,\eps h_1,\eps h_2)),(\eps h_1,\eps h_2)^{\otimes 2}\ra
\big].\label{dtsdehfy}
\end{multline}
Here $\nabla_{(x_1,x_2)}F(\gamma)$
 is defined analogously to $\nabla_x F(\gamma)$,
 \begin{equation}\label{hufdd}\tilde y(x_2-x_1,\eps(h_2-h_1)) \in [x_2-x_1,x_2-x_1+\eps(h_2-h_1)]\end{equation}
 and
 $$ \tilde\gamma(\gamma,x_1,x_2,\eps h_1,\eps h_2)=\gamma-\delta_{x_1}-\delta_{x_2}+\delta_{y_1}+\delta_{y_2}
 $$
 with
 $$(y_1,y_2)\in[(x_1,x_2),(x_1+\eps h_1,x_2+\eps  h_2)]. $$

 By condition a),
 \begin{equation}\label{tdtrs21} \int_{\R^d}dh\, a(h)h^i=0, \quad i=1,\dots,d.\end{equation}
 Hence, for any $\{x_1,\,x_2\}\subset\gamma$,
 \begin{align}
 &\eps^{-2}\int_{\R^d}dh_1\int_{\R^d}dh_2
\,a(h_1)a(h_2)2b(x_2-x_1) \la\nabla_{(x_1,x_2)}F(\gamma),(\eps h_1,\eps h_2)\ra\notag\\
&\quad=\eps^{-1}
2b(x_2-x_1)\left[\left\la \nabla_{x_1} F(\gamma),\int_{\R^d}dh_1\,a(h_1) h_1\int_{\R^d}dh_2\,
a(h_2) \right\ra\right.\notag\\
&\qquad\left.\text{}+\left\la\nabla_{x_2} F(\gamma),\int_{\R^d}dh_1\, a(h_1)\int_{\R^d}dh_2
\,a(h_2) h_2\right\ra\right]=0.\label{dseserjg}
 \end{align}

By condition a),
 \begin{equation}\label{tdtersare1} \int_{\R^d}dh\, a(h)h^ih^j=0, \quad i,j=1,\dots,d,\ i\ne j.\end{equation}
Now, by d), \eqref{trse}, \eqref{hufdd}, \eqref{tdtrs21},  \eqref{tdtersare1}, and the dominated convergence theorem, we get, as  $\eps\to0$,
\begin{align}
 &\eps^{-2}\int_{\R^d}dh_1\int_{\R^d}dh_2
\,a(h_1)a(h_2) \la \nabla b(\tilde y(x_2-x_1,\eps(h_2-h_1))),\eps(h_2-h_1)\ra\notag \\
&\qquad\times \la \nabla_{(x_1,x_2)}F(\gamma),(\eps h_1,\eps h_2)\ra\notag\\
&\quad= \int_{\R^d}dh_1\int_{\R^d}dh_2
\,a(h_1)a(h_2) \la \nabla b(\tilde y(x_2-x_1,\eps(h_2-h_1))),(h_2-h_1)\ra\notag \\
&\qquad\times \la \nabla_{(x_1,x_2)}F(\gamma),( h_1, h_2)\ra\notag\\
&\quad\to \int_{\R^d}dh_1\int_{\R^d}dh_2
\,a(h_1)a(h_2) \la \nabla b(x_2-x_1),(h_2-h_1)\ra \la \nabla_{(x_1,x_2)}F(\gamma),( h_1, h_2)\ra\notag\\
&\quad=\int_{\R^d}dh_1\int_{\R^d}dh_2
\,a(h_1)a(h_2)\sum_{i=1}^d\frac{\di}{\di y_i}b(y)\Big|_{y=x_2-x_1}(h_2^i-h_1^i)\notag\\
&\qquad\times\sum_{j=1}^d
\left(\frac{\di}{\di x_1^j}F(\gamma)h_1^j+\frac{\di}{\di x_2^j}F(\gamma)h_2^j\right)\notag\\
&\quad=c\big(\la \nabla_{x_1}F(\gamma),-\nabla b(x_2-x_1)\ra+\la \nabla_{x_2}F(\gamma),\nabla b(x_2-x_1)\ra\big)\notag\\
&\quad=c\big(\la \nabla_{x_1}F(\gamma),\nabla b(x_1-x_2)\ra+\la \nabla_{x_2}F(\gamma),\nabla b(x_2-x_1)\ra\big).\notag 
\end{align}
(We have used obvious notation.) Let $R>0$ be as in condition e).
Further, let $r>0$ be such that the function $a(h)$ vanishes outside the ball $B(0,r)$ and $\nabla_x F(\gamma)=0$ for all $\gamma\in \Gamma$ and $x\in\gamma$, $x\not\in B(0,r)$. Then,
for all $\eps<R/(2r)$,
\begin{align}
&\int_{\R^d}dh_1\int_{\R^d}dh_2
\,a(h_1)a(h_2) \big|\la \nabla b(\tilde y(x_2-x_1,\eps(h_2-h_1))),(h_2-h_1)\ra\big| \notag\\
&\qquad\times
\big|\la \nabla_{(x_1,x_2)}F(\gamma),( h_1, h_2)\ra\big|\notag\\
&\quad\le C_3 \int_{B(0,r)}dh_1 \int_{B(0,r)}dh_2(\mathbf 1_{B(0,r)}(x_1)+\mathbf 1_{B(0,r)}(x_2))
\sup_{y\in B(x_2-x_1,R)}|\nabla b(y)|(|h_1|^2+|h_2|^2)\notag\\
&\quad= C_4 (\mathbf 1_{B(0,r)}(x_1)+\mathbf 1_{B(0,r)}(x_2)) \sup_{y\in B(x_2-x_1,R)}|\nabla b(y)|. \label{yure565}
\end{align}
By \eqref{fdrssgdtd}, \eqref{yure565}, d) and e),
$$\sum_{\{x_1,\,x_2\}\subset\gamma} (\mathbf 1_{B(0,r)}(x_1)+\mathbf 1_{B(0,r)}(x_2)) \sup_{y\in B(x_2-x_1,R)}|\nabla b(y)| \in L^2(\Gamma,\pi_z).$$
Therefore, by the dominated convergence theorem,
\begin{align}
 &\sum_{\{x_1,x_2\} \subset\gamma}
 \eps^{-2}\int_{\R^d}dh_1\int_{\R^d}dh_2
\,a(h_1)a(h_2) \la \nabla b(\tilde y(x_2-x_1,\eps(h_2-h_1))),\eps(h_2-h_1)\ra\notag \\
&\qquad\times \la \nabla_{(x_1,x_2)}F(\gamma),(\eps h_1,\eps h_2)\ra\notag\\
&\quad\to \sum_{\{x_1,x_2\} \subset\gamma}c\big(\la \nabla_{x_1}F(\gamma),\nabla b(x_1-x_2)\ra+\la \nabla_{x_2}F(\gamma),\nabla b(x_2-x_1)\ra\big)\notag\\
&\quad= c\sum_{x\in\gamma}\left\la\nabla_x F(\gamma),\sum_{y\in\gamma\setminus\{x\}}b(x-y)\right\ra\label{cersr}
\end{align}
in $L^2(\Gamma,\pi_z)$ as $\eps\to0$.

Analogously,
\begin{align}
&\eps^{-2}\int_{\R^d}dh_1\int_{\R^d}dh_2
\,a(h_1)a(h_2) 2b(x_2-x_1)\notag\\
&\qquad\times (1/2)\la \nabla_{(x_1,x_2)}^2F(\tilde\gamma(\gamma,x_1,x_2,\eps h_1,\eps h_2)),(\eps h_1,\eps h_2)^{\otimes 2}\ra\notag\\
&\quad=\int_{\R^d}dh_1\int_{\R^d}dh_2
\,a(h_1)a(h_2) b(x_2-x_1)
\la \nabla_{(x_1,x_2)}^2F(\tilde\gamma(\gamma,x_1,x_2,\eps h_1,\eps h_2)),( h_1, h_2)^{\otimes 2}\ra\notag\\
&\quad\to c b(x_2-x_1)(\Delta_{x_1}F(\gamma)+\Delta_{x_2}F(\gamma))\notag
\end{align}
as $\eps\to 0$, and for $0<\eps\le 1$,
\begin{multline*}
\int_{\R^d}dh_1\int_{\R^d}dh_2
\,a(h_1)a(h_2) b(x_2-x_1)
\la \nabla_{(x_1,x_2)}^2F(\tilde\gamma(\gamma,x_1,x_2,\eps h_1,\eps h_2)),( h_1, h_2)^{\otimes 2}\ra\\
\le C_5 \big(\mathbf 1_{B(0,2r)}(x_1)+\mathbf 1_{B(0,2r)}(x_2)\big)b(x_2-x_1).
\end{multline*}
From here, by the dominated convergence,
\begin{align}
&\sum_{\{x_1,x_2\} \subset\gamma} \eps^{-2}\int_{\R^d}dh_1\int_{\R^d}dh_2
\,a(h_1)a(h_2) 2b(x_2-x_1)\notag\\
&\qquad\times (1/2)\la \nabla_{(x_1,x_2)}^2F(\tilde\gamma(\gamma,x_1,x_2,\eps h_1,\eps h_2)),(\eps h_1,\eps h_2)^{\otimes 2}\ra\notag\\
&\quad\to c\sum_{\{x_1,x_2\} \subset\gamma} b(x_2-x_1)(\Delta_{x_1}F(\gamma)+\Delta_{x_2}F(\gamma))\notag\\
&\quad=\sum_{x\in\gamma}\Delta_xF(\gamma)\sum_{y\in \gamma\setminus\{x\}}b(x-y).\label{frtst5}
\end{align}
in $L^2(\Gamma,\pi_z)$ as $\eps\to0$.

Finally, we easily conclude that
\begin{align}
&\sum_{\{x_1,x_2\} \subset\gamma} \eps^{-2}\int_{\R^d}dh_1\int_{\R^d}dh_2
\,a(h_1)a(h_2)\la \nabla b(\tilde y(x_2-x_1,\eps(h_2-h_1))),\eps(h_2-h_1)\ra\notag\\
&\qquad\times (1/2)\la \nabla_{(x_1,x_2)}^2F(\tilde\gamma(\gamma,x_1,x_2,\eps h_1,\eps h_2)),(\eps h_1,\eps h_2)^{\otimes 2}\ra\notag\\
&\quad =\eps \sum_{\{x_1,x_2\} \subset\gamma} \int_{\R^d}dh_1\int_{\R^d}dh_2
\,a(h_1)a(h_2)\la \nabla b(\tilde y(x_2-x_1,\eps(h_2-h_1))),h_2-h_1\ra\notag\\
&\qquad\times (1/2)\la \nabla_{(x_1,x_2)}^2F(\tilde\gamma(\gamma,x_1,x_2,\eps h_1,\eps h_2)),( h_1, h_2)^{\otimes 2}\ra\to 0\label{eopigf}
\end{align}
in $L^2(\Gamma,\pi_z)$ as $\eps\to0$.
Now, the statement of the theorem follows from \eqref{dtsdehfy}, \eqref{dseserjg},
\eqref{cersr}--\eqref{eopigf}.
\end{proof}

We will now show that the operator $(L_0,\FCC)$ in $L^2(\Gamma,\pi_z)$ from Theorem~\ref{ydsets} is a pre-generator of a diffusion dynamics. (The reader is advised to compare the following proposition with the paper \cite{AR}, which was the first rigorous result on the construction of an infinite dimensional diffusion through a Dirichlet form.)

\begin{proposition}\label{tesdl}  Let the conditions of Theorem~\rom{\ref{ydsets}} be satisfied.

\rom{i)} Define a quadratic form
$$ \mathcal E_0(F,G):=\int_\Gamma (-L_0F)(\gamma)G(\gamma)\pi_z(d\gamma),\quad F,G\in \FCC.$$
Then, for all $F,G\in\FCC$,
 \begin{align}\mathcal E_0(F,G)&=c\int_{\Gamma}\pi_z(d\gamma)\sum_{\{x_1,\,x_2\}\subset\gamma}b(x_1-x_2)
 \big(\la \nabla_{x_1}F(\gamma),\nabla_{x_1}G(\gamma)\ra+\la \nabla_{x_2}F(\gamma),\nabla_{x_2}G(\gamma)\ra\big)\notag\\
 &=c\int_{\Gamma}\pi_z(d\gamma)\sum_{x\in\gamma}\la \nabla_x F(\gamma),\nabla_x G(\gamma)\ra\sum_{y\in\gamma\setminus\{x\}}b(x-y) .\label{gzr}\end{align}
Hence, the quadratic form $(\mathcal E_0,\FCC)$ is symmetric and closable in $L^{2}(\Gamma,
\pi_z)$, and its closure will be denoted by $(\mathcal{E}_0,
D(\mathcal{E}_0))$.

\rom{ii)}  For $d\ge 2$, there exists a conservative diffusion process
\begin{align*}
M=\left(\Omega^0,
\mathcal{F}^0,(\mathcal{F}^0_{t})_{t\geq
0},(\Theta^0_{t})_{t\geq 0}, (X^0(t))_{t\geq 0},
(P^0_{\gamma})_{\gamma\in \Gamma}\right)
\end{align*}
on $\Gamma$ which is properly associated with
$(\mathcal{E}_0, D(\mathcal{E}_0))$, i.e., for each
($\pi_z$-version of) $F\in L^{2}(\Gamma, \pi_z)$ and $t>0$ $$\Gamma\ni
\gamma\mapsto
(p^0_{t}F)(\gamma):=\int_{\Omega}F(X^0(t))\, dP^0_{\gamma}$$
is an $\mathcal{E}_0$-quasi continuous version of
$\exp(tL_0)F$. Here  $(-L_0, D(L_0))$ is the
generator of the quadratic form $(\mathcal{E}_0, D(\mathcal{E}_0))$---the Friedrichs extension of the operator \linebreak $(-L_0,\FCC)$.

\rom{iii)} If $d=1$, then the result of\/ \rom{ii)} remains true if we replace $\Gamma$ with the bigger space $\ddot\Gamma$ of all multiple configurations in $\mathbb R^d$.
\end{proposition}

\begin{proof}
Analogously to the proof of Theorem \ref{ydsets}, we easily see that the quadratic form $({\mathcal E}_0',\FCC)$ given by the right-hand side of formula \eqref{gzr} is well defined. By the Mecke identity \eqref{Mecke},
\begin{align*}&{\mathcal E}'_0(F,G)=c\int_\Gamma\pi_z(d\gamma)\int_{\R^d}z\,dx\int_{\R^d}z\,dy\, \la\nabla_x F(\gamma+\delta_x+\delta_y),\nabla_x G(\gamma+\delta_x+\delta_y)\ra b(x-y)\\
&\quad=c\int_\Gamma\pi_z(d\gamma)\int_{\R^d}z\,dx\int_{\R^d}z\,dy\,
\big(-\Delta_x F(\gamma+\delta_x+\delta_y) b(x-y)\\
&\qquad\text{}-\la\nabla_x F(\gamma+\delta_x+\delta_y),\nabla b(x-y)\ra\big)G(\gamma+\delta_x+\delta_y)\\
&\quad=c\int_\Gamma\pi_z(d\gamma)\sum_{x\in\gamma}\sum_{y\in\gamma\setminus\{x\}}\big(-\Delta_x F(\gamma) b(x-y)
-\la\nabla_x F(\gamma),\nabla b(x-y)\ra\big)G(\gamma)\\
&\quad =\int_\Gamma (-L_0F)(\gamma)G(\gamma)\pi_z(d\gamma)=\mathcal E_0(F,G),\quad F,G\in \FCC.
\end{align*}
Thus, $(-L_0,\FCC)$ is the generator of the quadratic symmetric form \linebreak $(\mathcal E_0,\FCC)$ in $L^2(\Gamma,\pi_z)$.
Hence, this form is closable in $L^2(\Gamma,\pi_z)$, and so statement i) is proven. Statements ii) and iii) can be shown analogously to Theorems~6.1 and~6.3 in  \cite{KLRDiffusions}, see also \cite{MR98} and \cite{RS98}.
\end{proof}

A result similar to Theorem~\ref{ydsets} and Proposition~\ref{tesdl} can be obtained for the stochastic dynamics from Proposition~\ref{cdsesers}, ii). Let us briefly outline it. The scaled
$q$ function is given by
$$q_\eps (x,h):=\eps^{-d-2}a(h/\eps)b(x-h),$$
and let $L_\eps$ denote the corresponding  $L$ operator.
Hence,
\begin{multline*} (L_\eps F)(\gamma)=\sum_{\{x_1,\,x_2\}\subset\gamma}\eps^{-2}\int_{\R^d}dh\, a(h)b(x_2-x_1-\eps h)\\
\times
\big(F(\gamma\setminus\{x_1,x_2\}\cup\{x_1+\eps h,x_2-\eps h\})-F(\gamma)\big).\end{multline*}
Under the conditions of Theorem~\ref{ydsets},
for each $F\in\FCC$, $L_\eps F\to L_0F$ in $L^2(\Gamma,\pi_z)$  as $\eps\to0$, where
\begin{align*}
(L_0F)(\gamma):&= c \sum_{x\in\gamma}\left[\frac12\,\Delta_x F(\gamma)
\sum_{y\in\gamma\setminus \{x\}}b(x-y)+\left\la\nabla_x F(\gamma),\sum_{y\in\gamma\setminus\{x\}}\nabla b(x-y)\right\ra\right]\\
&\quad-c\sum_{\{x_1,\,x_2\}\subset\gamma}
b(x_2-x_1)\sum_{i=1}^d \frac{\di}{\di x_1^i}\, \frac{\di}{\di x_2^i}\,F(\gamma),
\end{align*}
$c$ being given by \eqref{trse}. The corresponding Dirichlet form $\mathcal E_0$ has the following representation on $\FCC$:
$$\mathcal E_0(F,G)=\frac c2\int_\Gamma\pi_z(d\gamma)\sum_{\{x_1,\,x_2\}\subset\gamma}b(x_2-x_1)
\big\la (\nabla_{x_1}-\nabla_{x_2})F(\gamma),  (\nabla_{x_1}-\nabla_{x_2})G(\gamma)\big\ra . $$

\section{Convergence to a birth-and-death process in continuum}\label{ufdey76ersr}

We will now consider
another scaling limit of  the stochastic dynamics as in Proposition~\ref{dersres}, which will lead us to a birth-and-death process in continuum. So, we now scale the dynamics as follows. For any $\eps>0$, we denote
$$ q_\eps(x,h_1,h_2):=\eps^{2d}a(\eps h_1)a(\eps h_2)\big(b(x)+b(x+h_2-h_1)\big),$$
and let $L_\eps$ denote the corresponding $L$ generator.
Hence, for each $F\in\FC$,
\begin{multline}
(L_\eps F)(\gamma)=\sum_{\{x_1,\,x_2\}\subset\gamma}\int_{\R^d}dh_1\int_{\R^d}dh_2\, a(h_1)
a(h_2)\big(b(x_2-x_1)+b(x_2-x_1+(h_2-h_1)/\eps)\big)\\
\times \big(F(\gamma\setminus\{x_1,x_2\}\cup\{x_1+(h_1/\eps),x_2+(h_2/\eps)\})-F(\gamma)\big).\label{fdse5weyre}
\end{multline}
It is not hard to show by approximation that, for any $\varphi\in C_0(\R^d)$, we have $e^{\la\varphi,\cdot\ra}\in D(L_\eps)$ and the action of $L_\eps$ on $F=e^{\la\varphi,\cdot\ra}$ is given by \eqref{fdse5weyre} (compare with the beginning of the proof of Theorem \ref{huigfmkih}).

Below, for a function $f\in L^1(\R^d,dx)$, we denote $\la f\ra:=\int_{\R^d}f(x)\,dx.$

\begin{theorem}\label{ds5tey} Let the conditions of Proposition \rom{\ref{dersres}} be satisfied. Additionally assume that
\begin{equation}\label{tsrtf} b(x)\to 0\quad\text{as }|x|\to\infty. \end{equation}
Then, for each $\varphi\in C_0(\R^d)$,
$$ \text{$L_\eps F\to L_0F$ in $L^2(\Gamma,\pi_z)$  as $\eps\to0$},$$
 where $F=e^{\la\varphi,\cdot\ra}$ and
 \begin{align}
 &(L_0F)(\gamma):=\la a\ra^2\left[
 \sum_{\{x_1,\,x_2\}\subset\gamma}b(x_2-x_1)\big(F(\gamma\setminus\{x_1,x_2\})-F(\gamma)\big)\right.\notag\\
 &\qquad\text{}+\frac12
 \int_{\R^d}z\,dx_1\, \int_{\R^d}z\,dx_2 \, b(x_2-x_1)\big(F(\gamma\cup\{x_1,x_2\})-F(\gamma)\big)\notag\\
 &\qquad\left.+z\la b\ra \sum_{x\in\gamma}\big(F(\gamma\setminus\{x\})-F(\gamma)\big)+z\la b\ra \int_{\R^d}
z\, dx\big(F(\gamma\cup\{x\})-F(\gamma)\big)\right].\label{dreast}
 \end{align}
\end{theorem}

\begin{proof}
We represent $L_\eps F$ as follows:
$$ (L_\eps F)(\gamma)=(L^{(1)}_0 F)(\gamma)+\sum_{i=2}^4 (L^{(i)}_\eps F)(\gamma),$$
where
\begin{align*}
&(L^{(1)}_0 F)(\gamma):= \la a\ra^2\sum_{\{x_1,\,x_2\}\subset\gamma}b(x_2-x_1)\big(F(\gamma\setminus\{x_1,x_2\})-F(\gamma)\big),\\
&(L^{(2)}_\eps F)(\gamma):= \sum_{\{x_1,\,x_2\}\subset\gamma}\int_{\R^d}dh_1\int_{\R^d}dh_2\, a(h_1)a(h_2)
b(x_2-x_1+(h_2-h_1)/\eps))\\
&\qquad\times\big(F(\gamma\setminus\{x_1,x_2\})-F(\gamma)\big),\\
& (L^{(3)}_\eps F)(\gamma):=\sum_{\{x_1,\,x_2\}\subset\gamma} \int_{\R^d}dh_1\int_{\R^d}dh_2\, a(h_1)a(h_2) b(x_2-x_1)\\
&\qquad\times \big(
F(\gamma\setminus\{x_1,x_2\}\cup\{x_1+(h_1/\eps),\, x_2+(h_2/\eps)\})-F(\gamma\setminus\{x_1,x_2\})
\big),\\
& (L^{(4)}_\eps F)(\gamma):=\sum_{\{x_1,\,x_2\}\subset\gamma} \int_{\R^d}dh_1\int_{\R^d}dh_2\, a(h_1)a(h_2) b(x_2-x_1+(h_2-h_1)/\eps)\\
&\qquad\times \big(
F(\gamma\setminus\{x_1,x_2\}\cup\{x_1+(h_1/\eps),\, x_2+(h_2/\eps)\})-F(\gamma\setminus\{x_1,x_2\})
\big).
\end{align*}
The statement of the theorem will follow if we show that, for each
$F=e^{\la\varphi,\cdot\ra}$, $\varphi\in C_0(\R^d)$, and  $i=2,3,4$,
\begin{align} \int_\Gamma\pi_z(d\gamma)(L^{(i)}_\eps F)^2(\gamma)&\to
\int_\Gamma\pi_z(d\gamma)(L^{(i)}_0F)^2(\gamma )\quad\text{as $\eps\to0$},\label{jfdrstty}\\
 \int_\Gamma\pi_z(d\gamma)(L^{(i)}_\eps F)(\gamma)(L^{(i)}_0F)(\gamma )&\to
\int_\Gamma\pi_z(d\gamma)(L^{(i)}_0F)^2(\gamma )\quad\text{as $\eps\to0$},\label{ds5u}\end{align}
where
\begin{align*}
&(L_0^{(2)}F(\gamma):=\la a\ra^2z\la b\ra \sum_{x\in\gamma}\big(F(\gamma\setminus\{x\})-F(\gamma)\big),\\
&(L_0^{(3)}F)(\gamma):=\la a\ra^2z\la b\ra \int_{\R^d}
z\, dx\big(F(\gamma\cup\{x\})-F(\gamma)\big),\\
&(L_0^{(4)}F)(\gamma):=\la a\ra^2 \,\frac12
 \int_{\R^d}z\,dx_1\, \int_{\R^d}z\,dx_2 \, b(x_2-x_1)\big(F(\gamma\cup\{x_1,x_2\})-F(\gamma)\big).
\end{align*}
In fact, we have, for $\eps>0$,
\begin{align*}
&(L^{(2)}_\eps F)(\gamma):=e^{\la\varphi,\gamma\ra} \sum_{\{x_1,\,x_2\}\subset\gamma}\int_{\R^d}dh_1\int_{\R^d}dh_2\, a(h_1)a(h_2)
b(x_2-x_1+(h_2-h_1)/\eps))\\
&\qquad\times\big(e^{-\varphi(x_1)-\varphi(x_2)}-1\big),\\
& (L^{(3)}_\eps F)(\gamma):=e^{\la\varphi,\gamma\ra} \sum_{\{x_1,\,x_2\}\subset\gamma} \int_{\R^d}dh_1\int_{\R^d}dh_2\, a(h_1)a(h_2) b(x_2-x_1)\\
&\qquad\times e^{-\varphi(x_1)-\varphi(x_2)}\big(
e^{\varphi(x_1+(h_1/\eps))+\varphi(x_2+(h_2/\eps))}-1
\big),\\
& (L^{(4)}_\eps F)(\gamma):=e^{\la\varphi,\gamma\ra}\sum_{\{x_1,\,x_2\}\subset\gamma} \int_{\R^d}dh_1\int_{\R^d}dh_2\, a(h_1)a(h_2) b(x_2-x_1+(h_2-h_1)/\eps)\\
&\qquad\times  e^{-\varphi(x_1)-\varphi(x_2)}\big(
e^{\varphi(x_1+(h_1/\eps))+\varphi(x_2+(h_2/\eps))}-1
\big).
\end{align*}
and
\begin{align*}
&(L_0^{(2)}F)(\gamma):=\la a\ra^2z\la b\ra\, e^{\la\varphi,\gamma\ra} \sum_{x\in\gamma}(e^{-\varphi(x)}-1),\\
&(L_0^{(3)}F)(\gamma):=\la a\ra^2z\la b\ra
\, e^{\la\varphi,\gamma\ra}
\int_{\R^d}
z\, dx\,(e^{\varphi(x)}-1),\\
&(L_0^{(4)}F)(\gamma):=\frac12\,\la a\ra^2  e^{\la\varphi,\gamma\ra}
 \int_{\R^d}z\,dx_1\, \int_{\R^d}z\,dx_2 \, b(x_2-x_1)(e^{\varphi(x_1)+\varphi(x_2)}-1).
\end{align*}

Analogously to \eqref{fdrssgdtd}, we get
\begin{align}
&\int_\Gamma\pi_z(d\gamma)\left(e^{\la\varphi,\gamma\ra}\sum_{\{x_1,\,x_2\}\subset\gamma}f(x_1,x_2)
\right)^2\notag\\
&=\int_\Gamma\pi_z(d\gamma)e^{\la2\varphi,\gamma\ra}\left[\frac14\left(\int_{\R^d}z\,dx_1\int_{\R^d}z\,dx_2\, e^{2\varphi(x_1)+2\varphi(x_2)}f(x_1,x_2)\right)^2\right.\notag\\
&\quad\text{}+
\int_{\R^d}z\,dx_1\int_{\R^d}z\,dx_2\int_{\R^d}z\,dx_3\, e^{2\varphi(x_1)+2\varphi(x_2)+2\varphi(x_3)}f(x_1,x_2)f(x_2,x_3)\notag\\
&\quad\left. \text{}+\frac12\int_{\R^d}z\,dx_1\int_{\R^d}z\,dx_2\,
e^{2\varphi(x_1)+2\varphi(x_2)}
f(x_1,x_2)^2
\vphantom{\frac14\left(\int_{\R^d}z\,dx_1\int_{\R^d}z\,dx_2\, e^{2\varphi(x_1)+2\varphi(x_2)}f(x_1,x_2)\right)^2}
\right]
\label{ghdstre}\end{align}
for any measurable function $f:(\mathbb R^d)^2\to[0,\infty]$ and any $\varphi\in C_0(\R^d)$.

Let us show that \eqref{jfdrstty} holds for $i=2$. Since
$$ e^{-\varphi(x_1)-\varphi(x_2)}-1=e^{-\varphi(x_1)}(e^{-\varphi(x_2)}-1)+(e^{-\varphi(x_1)}-1),$$
by  the dominated convergence theorem, we have
\begin{align}&
\frac14\int_\Gamma \pi_z(d\gamma)e^{\la 2\varphi,\gamma\ra}\left(\int_{\R^d}z\,dx_1\int_{\R^d}z\,dx_2
\int_{\R^d}dh_1 \int_{\R^d}dh_2\, a(h_1)a(h_2)e^{2\varphi(x_1)+2\varphi(x_2)}\notag\right.\\
&\qquad\times\left.\vphantom{\int_{\R^d}}
b(x_2-x_1+(h_2-h_1)/\eps)(e^{-\varphi(x_1)-\varphi(x_2)}-1)\right)^2\notag\\
&\quad =\frac14\int_\Gamma \pi_z(d\gamma)e^{\la 2\varphi,\gamma\ra}\left(\int_{\R^d}z\,dx_1\int_{\R^d}z\,dx_2
\int_{\R^d}dh_1 \int_{\R^d}dh_2\, a(h_1)a(h_2)  b(x_2-x_1)\notag\right.\\
&\qquad\times e^{\varphi(x_1-(h_1/\eps)+(h_2/\eps))+2\varphi(x_2)}(e^{-\varphi(x_2)}-1)\notag\\
&\qquad\text{}+\int_{\R^d}z\,dx_1\int_{\R^d}z\,dx_2
\int_{\R^d}dh_1 \int_{\R^d}dh_2\, a(h_1)a(h_2)b(x_2-x_1) \notag\\
&\qquad\times \left.
e^{2\varphi(x_1)+2\varphi(x_2-(h_2/\eps)+(h_1/\eps))}
(e^{-\varphi(x_1)}-1)\vphantom{\int_{\R^d}}\right)^2\notag\\
&\quad\to \int_\Gamma\pi_z(d\gamma)e^{\la 2\varphi,\gamma\ra}\left(\la a\ra^2z\,\la b\ra\int_{\R^d}z\,dx\, e^{2\varphi(x)}(e^{-\varphi(x)}-1)\right)^2
\quad\text{as }\eps\to0.\label{dryde}
\end{align}
Analogously,
\begin{align}&
\int_\Gamma \pi_z(d\gamma)e^{\la 2\varphi,\gamma\ra}\int_{\R^d}z\,dx_1\int_{\R^d}z\,dx_2
\int_{\R^d}z\,dx_3\int_{\R^d}dh_1 \int_{\R^d}dh_2\int_{\R^d}dh'_1 \int_{\R^d}dh'_2\, \notag\\
&\qquad\times a(h_1)a(h_2)a(h'_1)a(h'_2) e^{2\varphi(x_1)+2\varphi(x_2)+2\varphi(x_3)}\notag\\
&\qquad\times b(x_2-x_1+(h_2-h_1)/\eps)
b(x_3-x_2+(h'_2-h'_1)/\eps)\notag\\
&\qquad\times\big(e^{-\varphi(x_1)}(e^{-\varphi(x_2)}-1)+ (e^{-\varphi(x_1)}-1)\big)\big(
e^{-\varphi(x_2)}(e^{-\varphi(x_3)}-1)+(e^{-\varphi(x_2)}-1)\big)\notag\\
&\quad\to \int_\Gamma \pi_z(d\gamma)e^{\la 2\varphi,\gamma\ra}\int_{\R^d}z\,dx_1\int_{\R^d}z\,dx_2
\int_{\R^d}z\,dx_3\int_{\R^d}dh_1 \int_{\R^d}dh_2\int_{\R^d}dh'_1 \int_{\R^d}dh'_2\, \notag\\
&\qquad\times a(h_1)a(h_2)a(h'_1)a(h'_2) e^{2\varphi(x_2)}b(x_2-x_1)b(x_3-x_2)(e^{-\varphi(x_2)}-1)^2\notag\\
&\quad=\la a\ra^4z^2\la b\ra^2\int_\Gamma \pi_z(d\gamma)e^{\la 2\varphi,\gamma\ra}\int_{\R^d}z\,dx\, e^{2\varphi(x)}(e^{-\varphi(x)}-1)^2,
\end{align}
and using additionally \eqref{tsrtf},
\begin{align}&
\frac12\int_\Gamma \pi_z(d\gamma)e^{\la 2\varphi,\gamma\ra}\int_{\R^d}z\,dx_1\int_{\R^d}z\,dx_2
\int_{\R^d}dh_1 \int_{\R^d}dh_2\int_{\R^d}dh'_1 \int_{\R^d}dh'_2\, \notag\\
&\qquad\times a(h_1)a(h_2)a(h'_1)a(h'_2) e^{2\varphi(x_1)+2\varphi(x_2)}\notag\\
&\qquad\times b(x_2-x_1+(h_2-h_1)/\eps))b(x_2-x_1+(h'_2-h_1')/\eps))\notag\\
&\qquad\times
\big[e^{-\varphi(x_1)}(e^{-\varphi(x_2)}-1)+(e^{-\varphi(x_1)}-1)\big]^2\to0\quad\text{as }\eps\to0.\label{ysa}
\end{align}
By \eqref{dryde}--\eqref{ysa},  formula \eqref{jfdrstty}  for $i=2$  follows.

Next, we show \eqref{jfdrstty}  for $i=3$. Analogously to the above, we get, as $\eps\to0$,
\begin{align}
&\frac14\int_\Gamma\pi_z(d\gamma)e^{\la 2\varphi,\gamma\ra}\left(\int_{\R^d}z\,dx_1
 \int_{\R^d}z\,dx_2 \int_{\R^d}dh_1 \int_{\R^d}dh_2\, a(h_1)a(h_2)b(x_2-x_1)\right.\notag\\
 &\qquad\left. \vphantom{\int_\Gamma \int_{\R^d}}\times e^{2\varphi(x_1)+2\varphi(x_2)-\varphi(x_1)-\varphi(x_2)}(e^{\varphi(x_1+(h_1/\eps))+\varphi(x_2+(h_2/\eps))}-1)\right)^2\notag\\
 &\quad=\frac14\int_\Gamma\pi_z(d\gamma)e^{\la 2\varphi,\gamma\ra}\left(\int_{\R^d}z\,dx_1
 \int_{\R^d}z\,dx_2 \int_{\R^d}dh_1 \int_{\R^d}dh_2\, a(h_1)a(h_2)\right.\notag\\
 &\qquad\left. \vphantom{\int_\Gamma \int_{\R^d}}\times b(x_2-(h_2/\eps)-x_1+(h_1/\eps)) e^{\varphi(x_1-(h_1/\eps))+\varphi(x_2-(h_2/\eps))}
 (e^{\varphi(x_1)+\varphi(x_2)}-1)\right)^2\notag\\
 &\quad\to\frac14\int_\Gamma\pi_z(d\gamma)e^{\la 2\varphi,\gamma\ra}\left(\int_{\R^d}z\,dx_1
 \int_{\R^d}z\,dx_2 \int_{\R^d}dh_1 \int_{\R^d}dh_2\, a(h_1)a(h_2)b(x_2-x_1)\right.\notag\\
 &\qquad\left. \vphantom{\int_\Gamma \int_{\R^d}}\times \big((e^{\varphi(x_1)}-1)+(e^{\varphi(x_2)}-1)\big)\right)^2\notag\\
 &\quad=\la a\ra^4 z^2\la b\ra^2\int_\Gamma\pi_z(d\gamma)e^{\la 2\varphi,\gamma\ra}\left(\int_{\R^d}dx\,(e^{\varphi(x)}-1)\right)^2,\label{fdsop}\\
  & \int_\Gamma\pi_z(d\gamma)e^{\la 2\varphi,\gamma\ra}
  \int_{\R^d}z\,dx_1
 \int_{\R^d}z\,dx_2 \int_{\R^d}z\,dx_3\int_{\R^d}dh_1 \int_{\R^d}dh_2 \int_{\R^d}dh'_1 \int_{\R^d}dh'_2\notag\\
 &\qquad\times a(h_1)a(h_2)a(h_1')a(h_2') b(x_2-x_1)b(x_3-x_2) e^{2\varphi(x_1)+2\varphi(x_2)+2\varphi(x_3)-\varphi(x_1)-2\varphi(x_2)-\varphi(x_3)}\notag\\
 &\qquad\times (e^{\varphi(x_1+(h_1/\eps))+\varphi(x_2+(h_2/\eps))}-1)(e^{\varphi(x_2+(h_1'/\eps))+\varphi(x_3+(h_2'/\eps))}-1)
 \notag\\
 & \quad=\int_\Gamma\pi_z(d\gamma)e^{\la 2\varphi,\gamma\ra}
  \int_{\R^d}z\,dx_1
 \int_{\R^d}z\,dx_2 \int_{\R^d}z\,dx_3\int_{\R^d}dh_1 \int_{\R^d}dh_2 \int_{\R^d}dh'_1 \int_{\R^d}dh'_2\notag\\
 &\qquad\times a(h_1)a(h_2)a(h_1')a(h_2')b(x_2-x_1) b(x_3-x_2) e^{\varphi(x_1)+\varphi(x_3)}\notag\\
 &\qquad\times (e^{\varphi(x_1+(h_1/\eps))+\varphi(x_2+(h_2/\eps))}-1)\big[
 e^{\varphi(x_2+(h_1'/\eps))}(e^{\varphi(x_3+(h_2'/\eps))}-1)+( e^{\varphi(x_2+(h_1'/\eps))}-1)\big]
 \notag\\
 &\quad=\int_\Gamma\pi_z(d\gamma)e^{\la 2\varphi,\gamma\ra}
  \int_{\R^d}z\,dx_1
 \int_{\R^d}z\,dx_2 \int_{\R^d}z\,dx_3\int_{\R^d}dh_1 \int_{\R^d}dh_2 \int_{\R^d}dh'_1 \int_{\R^d}dh'_2\notag\\
 &\qquad\times a(h_1)a(h_2)a(h_1')a(h_2')b(x_2-x_1) b(x_3-x_2) \notag\\
 &\qquad\times\big[e^{\varphi(x_1)+\varphi(x_3-(h_2'/\eps))}
 (e^{\varphi(x_1+(h_1/\eps))+\varphi(x_2+(h_2/\eps))}-1)
 e^{\varphi(x_2+(h_1'/\eps))}(e^{\varphi(x_3)}-1)\notag\\
 &\qquad\text{}
 +e^{\varphi(x_1)+\varphi(x_3)}
 (e^{\varphi(x_1+(h_1/\eps))+\varphi(x_2+(h_2/\eps)-(h_1'/\eps))}-1)
(e^{\varphi(x_2)}-1)\big]\to0,\label{dsreyu}\\
& \frac12\int_\Gamma\pi_z(d\gamma)e^{\la 2\varphi,\gamma\ra}
  \int_{\R^d}z\,dx_1
 \int_{\R^d}z\,dx_2 \int_{\R^d}dh_1 \int_{\R^d}dh_2 \int_{\R^d}dh'_1 \int_{\R^d}dh'_2\notag\\
 &\qquad\times a(h_1)a(h_2)a(h_1')a(h_2')b(x_2-x_1)^2 e^{2\varphi(x_1)+2\varphi(x_2)-2\varphi(x_1)-2\varphi(x_2)}\notag\\
 &\qquad\times
 (e^{\varphi(x_1+(h_1/\eps))+\varphi(x_2+(h_2/\eps))}-1)
 (e^{\varphi(x_1+(h'_1/\eps))+\varphi(x_2+(h'_2/\eps))}-1)\to 0.\label{vcxd}
 \end{align}
By \eqref{fdsop}--\eqref{vcxd}, formula  \eqref{jfdrstty}  for $i=3$ follows.

Now, we show \eqref{jfdrstty}  for $i=4$. Similarly to the above, we get:
\begin{align}
&\frac14\int_\Gamma\pi_z(d\gamma)e^{\la 2\varphi,\gamma\ra}\left(\int_{\R^d}z\,dx_1
 \int_{\R^d}z\,dx_2 \int_{\R^d}dh_1 \int_{\R^d}dh_2\, a(h_1)a(h_2)\right.\notag\\
 &\qquad\left. \vphantom{\int_\Gamma \int_{\R^d}}
 \times b(x_2-x_1+(h_2-h_1)/\eps)) e^{\varphi(x_1)+\varphi(x_2)}
 (e^{\varphi(x_1+(h_1/\eps))+\varphi(x_2+(h_2/\eps))}-1)\right)^2\notag\\
&\quad=\frac14\int_\Gamma\pi_z(d\gamma)e^{\la 2\varphi,\gamma\ra}\left(\int_{\R^d}z\,dx_1
 \int_{\R^d}z\,dx_2 \int_{\R^d}dh_1 \int_{\R^d}dh_2\, a(h_1)a(h_2)\right.\notag\\
 &\qquad\left. \vphantom{\int_\Gamma \int_{\R^d}}
 \times  b(x_2-x_1) e^{\varphi(x_1-(h_1/\eps))+\varphi(x_2-(h_2/\eps))}
  (e^{\varphi(x_1)+\varphi(x_2))}-1)\right)^2\notag\\
 &\quad\to \frac14\int_\Gamma\pi_z(d\gamma)e^{\la 2\varphi,\gamma\ra}\left(\int_{\R^d}z\,dx_1
 \int_{\R^d}z\,dx_2 \int_{\R^d}dh_1 \int_{\R^d}dh_2\, a(h_1)a(h_2)\right.\notag\\
 &\qquad\left. \vphantom{\int_\Gamma \int_{\R^d}}
 \times
 b(x_2-x_1) (e^{\varphi(x_1)+\varphi(x_2))}-1)\right)^2\notag\\
&\quad=\frac14\int_\Gamma\pi_z(d\gamma)e^{\la 2\varphi,\gamma\ra} \la a\ra ^4\left(
\int_{\R^d}z\,dx_1
 \int_{\R^d}z\,dx_2 b(x_2-x_1) (e^{\varphi(x_1)+\varphi(x_2))}-1)
\right)^2,\label{asubhyf}\\
 & \int_\Gamma\pi_z(d\gamma)e^{\la 2\varphi,\gamma\ra}
  \int_{\R^d}z\,dx_1
 \int_{\R^d}z\,dx_2 \int_{\R^d}z\,dx_3\int_{\R^d}dh_1 \int_{\R^d}dh_2 \int_{\R^d}dh'_1 \int_{\R^d}dh'_2\notag\\
 &\qquad\times a(h_1)a(h_2)a(h_1')a(h_2') b(x_2+(h_2/\eps)-x_1-(h_1/\eps))b(x_3+(h'_2/\eps)-x_2-(h'_1/\eps))\notag\\
 &\qquad \times e^{\varphi(x_1)+\varphi(x_3)} (e^{\varphi(x_1+(h_1/\eps))+\varphi(x_2+(h_2/\eps))}-1)
 (e^{\varphi(x_2+(h'_1/\eps))+\varphi(x_3+(h'_2/\eps))}-1)\to 0,\label{dstyugf}\\
 & \frac12\int_\Gamma\pi_z(d\gamma)e^{\la 2\varphi,\gamma\ra}
  \int_{\R^d}z\,dx_1
 \int_{\R^d}z\,dx_2 \int_{\R^d}dh_1 \int_{\R^d}dh_2 \int_{\R^d}dh'_1 \int_{\R^d}dh'_2\notag\\
 &\qquad\times a(h_1)a(h_2)a(h_1')a(h_2')b(x_2+(h_2/\eps)-x_1-(h_1/\eps))b(x_2-x_1+(h'_2-h'_1)/\eps))\notag\\
 &\qquad\times (e^{\varphi(x_1+(h_1/\eps))+\varphi(x_2+(h_2/\eps))}-1)
 (e^{\varphi(x_1+(h'_1/\eps))+\varphi(x_2+(h'_2/\eps))}-1)\to0. \label{jxsretf}
 \end{align}
By \eqref{asubhyf}--\eqref{jxsretf}, formula  \eqref{jfdrstty}  for $i=4$ follows. Thus, \eqref{jfdrstty} is proven. Formula \eqref{ds5u} follows analogously.
\end{proof}

Denote by $\mathscr F_{\mathrm{exp}}$ the linear span of $\{e^{\la\varphi,\cdot\ra}, \ \varphi\in C_0(\mathbb R^d)\}$. Consider a linear operator $(L_0,\mathscr F_{\mathrm{exp}})$ where, for each $F\in \mathscr F_{\mathrm{exp}}$, $L_0F$ is given by \eqref{dreast}. Analogously to Theorem~\ref{e6ue6}, Proposition \ref{dersres}, and Theorem~\ref{huigfmkih}, we get

\begin{proposition}\label{dxsrtsd} \rom{i)} Let the conditions of Proposition \rom{\ref{dersres}} be satisfied.
Define a quadratic form
$$ \mathcal E_0(F,G):=\int_\Gamma (-L_0F)(\gamma)G(\gamma)\pi_z(d\gamma),\quad F,G\in \mathscr F_{\mathrm{exp}}.$$
Then, for all $F,G\in\mathscr F_{\mathrm{exp}}$,
 \begin{align}&\mathcal E_0(F,G)=\la a\ra^2\int_\Gamma \pi_z(d\gamma)\left[
 \sum_{\{x_1,\,x_2\}\subset\gamma}b(x_2-x_1)\notag\right.\\
 &\qquad \times\big(F(\gamma\setminus\{x_1,x_2\})-F(\gamma)\big)
\big (G(\gamma\setminus\{x_1,x_2\})-G(\gamma)\big)\notag\\
 &\qquad \left. \vphantom{\pi_z(d\gamma)
 \sum_{\{x_1,\,x_2\}\subset\gamma}}\quad+z\la b\ra\sum_{x\in\gamma}\big(F(\gamma\setminus\{x\})-F(\gamma)\big)
 \big(G(\gamma\setminus\{x\})-G(\gamma)\big)\right].\label{dtsjkgy}
  \end{align}
Hence, the quadratic form $(\mathcal E_0,\mathscr F_{\mathrm{exp}})$ is symmetric and closable in $L^{2}(\Gamma,
\pi_z)$, and its closure will be denoted by $(\mathcal{E}_0,
D(\mathcal{E}_0))$.
Further there exists a conservative Hunt process
\begin{align*}
M_0=\left(\Omega^0,
\mathcal{F}^0,(\mathcal{F}^0_{t})_{t\geq
0},(\Theta^0_{t})_{t\geq 0}, (X^0(t))_{t\geq 0},
(P^0_{\gamma})_{\gamma\in \Gamma}\right)
\end{align*}
on $\Gamma$ which is properly associated with
$(\mathcal{E}_0, D(\mathcal{E}_0))$.

\rom{ii)} The operator $(-L_0,\mathscr F_{\mathrm{exp}})$  is essentially selfadjoint in $L^2(\Gamma,\pi_z)$.

\end{proposition}

Denote by
$$
M_\eps=\left(\Omega^\eps,
\mathcal{F}^\eps,(\mathcal{F}^\eps_{t})_{t\geq
0},(\Theta^\eps_{t})_{t\geq 0}, (X^\eps(t))_{t\geq 0},
(P^\eps_{\gamma})_{\gamma\in \Gamma}\right)
$$
the Markov processes from Theorem~\ref{e6ue6} which corresponds to the $L_\eps$ generator.
By the theory of Dirichlet forms \cite{MR}, the Markov processes $M_\eps$, as well as the Markov process $M_0$ from Proposition~\ref{dxsrtsd} can be chosen in the canonical form, i.e., for each $\eps\ge0$, $\Omega^\eps$ is the set
$D(\left[0,+\infty\right),\Gamma)$ of all {\it c\'adl\'ag\/} functions
$\omega:\left[0,+\infty\right)\to\Gamma$ (i.e., $\omega$
is right continuous on $\left[0,+\infty\right)$ and has
left limits on $(0,+\infty)$),
$X^\eps(t)(\omega)=\omega(t)$, $t\geq0$,
$\omega\in\Omega^\eps$, $(\mathcal F^\eps_t)_{t\geq0}$
together with $\mathcal F^\eps$ is the corresponding minimum
completed admissible family (cf.\ \cite[Section 4.1]{Fu80}) and $\Theta^\eps_t$, $t\geq0$, are the corresponding
natural time shifts. So, for each $\eps\ge0$, we choose the canonical version of the $M^\eps$ process and define a stochastic process $Y_\eps=(Y_\eps(t))_{t\ge0}$ whose law is the probability measure on $D(\left[0,+\infty\right),\Gamma)$ given by $Q_\eps:=\int_\Gamma\pi_z(d\gamma) P_\gamma^\eps$. Note that $\pi_z$ is an invariant measure for $Y_\eps$.

\begin{corollary} Let the conditions of Theorem \rom{\ref{ds5tey}} be satisfied. Then the finite-dim\-ens\-ional distributions of the process $Y_\eps$ weakly converge to the finite-dimensional distributions of $Y_0$ as $\eps\to0$.
\end{corollary}

\begin{proof}
The statement follows analogously to \cite[Theorem~5.1]{FKL}, however, since the argumentation is rather short, we present it. By Theorem~\ref{ds5tey}, Proposition~\ref{dxsrtsd}, ii) and \cite[Chapter~3, Theorem~3.17]{Dav},  we have, for each $t\ge0$,
$e^{tL_\varepsilon}\to e^{tL_0}$ strongly in $L^2(\Gamma,\pi_z)$
as $\varepsilon\to0$. We now fix any $0\le t_1<t_2<\dots<t_n$, $n\in\mathbb N$. For $\eps\ge0$, denote by $\mu_{t_1,\dots,t_n}^\eps$ the
finite-dimensional  distribution of the process  $Y_\eps$ at times $t_1,\dots,t_n$, which is a probability measure on $\Gamma^n$.
Since $\Gamma$ is a Polish space, by \cite[Chapter II, Theorem~3.2]{Par}, the  measure $\pi_z$ is tight on $\Gamma$. Since all the marginal distributions of the measure $\mu_{t_1,\dots,t_n}^\eps$ are $\pi_z$, we therefore conclude that the set $\{\mu_{t_1,\dots,t_n}^\eps\mid\eps>0\}$ is pre-compact in the space $\mathcal M(\Gamma^n)$ of the probability measures on $\Gamma^n$ with respect to the weak topology, see e.g.\ \cite[Chapter II, Section~6]{Par}. Hence, the weak convergence of finite-dimensional distributions follows from the strong convergence of the semigroups.
\end{proof}

\begin{remark}
The dynamics as in Proposition \ref{cdsesers}, ii)  can be scaled as follows:
$$ q_\eps(x,h):=\eps^d a(\eps h)b(x-h).$$
By analogy, one can show that the corresponding dynamics converge, as $\eps\to0$, to a birth-and-death process in continuum with generator
$$ (L_0F)(\gamma)=\la a\ra\la b\ra \left(\sum_{x\in\gamma}\big(F(\gamma\setminus\{x\})-F(\gamma)\big)+\int_{\R^d}z\,dx\big(
F(\gamma\cup\{x\})-F(\gamma)\big)\right).$$
The operator $L_0$, realized in the Fock space $\mathcal F(L^2(\R^d,z\,dx))$, is the differential second quantization of the operator $\la a\ra\la b\ra \mathbf 1$, so the corresponding dynamics is `free', i.e., without interaction between particles, see \cite{KLR_free,Surgailis,Surgailis2} for further detail.
\end{remark}

\begin{corollary}\label{uiyre6} The quadratic form $(\mathcal E_0,D(\mathcal E_0))$ from Proposition \ref{dxsrtsd}, i) satisfies
the Poincar\'e inequality:
\begin{equation}\label{uyfde}\mathcal E_0(F,F)\ge\la a\ra^2z\la b\ra \int_\Gamma (F(\gamma)-\la F\ra_{\pi_z})^2\pi_z(d\gamma),\quad F\in D(\mathcal E_0),\end{equation}
where $\la F\ra_{\pi_z}:=\int_\Gamma F(\gamma)\pi_z(d\gamma)$.
\end{corollary}

\begin{remark} The Poincar\'e inequality means that the operator $(-L_0,D(L_0))$ has a spectral gap, the set $\big(0,\la a\ra^2z\la b\ra \big)$, and that the kernel of $(-L_0,D(L_0))$ consists only of the constants.

\end{remark}

\begin{proof}
Recall the set $\mathscr P$ from the proof of Theorem \ref{huigfmkih}. Clearly, $\mathscr P$ is a core for the quadratic form $(\mathcal E_0,D(\mathcal E_0))$, so it suffices to prove \eqref{uyfde} only for any $F\in\mathscr P$. By \eqref{dtsjkgy},
\begin{equation}\label{hgtyd}
\mathcal E_0(F,F)\ge \la a\ra^2z\la b\ra\, \mathcal E_0'(F,F),\quad F\in\mathscr P,\end{equation}
where
$$
\mathcal E_0'(F,G):=\int_\Gamma\big( F(\gamma\setminus\{x\})-F(\gamma)\big) \big( G(\gamma\setminus\{x\})-G(\gamma)\big) \pi_z(d\gamma),\quad F,G\in\mathscr P.
$$
The generator of the quadratic form $(\mathcal E_0',D(\mathcal E_0'))$, realized in the Fock space $\mathcal F(L^2(\R^d,z\,dx))$, has a representation $\int_{\R^d}z\,dx\,\di^\dag_x\di_x$, i.e., it is the differential second quantization of the identity operator $\mathbf 1$, i.e., for any $f^{(n)}\in\mathcal F^{(n)}(L^2(\R^d.z\,dx))$,
$$ \bigg(\int_{\R^d}z\,dx\,\di^\dag_x\di_xf^{(n)}\bigg)(y_1,\dots,y_n)=nf^{(n)}(y_1,\dots,y_n).$$
 Hence,
\begin{equation}\label{ldrtuygty}
\mathcal E_0'(F,F)\ge \int_\Gamma (F(\gamma)-\la F\ra_{\pi_z})^2\pi_z(d\gamma),\quad F\in\mathscr P.
\end{equation}
The statement now follows from \eqref{hgtyd} and \eqref{ldrtuygty}.
\end{proof}

\begin{remark} We note that the initial dynamics of binary jumps is translation invariant and conservative. So it is hopeless to expect that its generator has a spectral gap. So the spectral gap of the generator $L_0$ appears as a result of the scaling limit.

Note also that the generator of the dynamics of binary jumps is independent
of the intensity parameter  $z>0$. Hence, at least heuristically,
the initial  dynamics has a continuum of symmetrizing Poisson  measures, indexed by the intensity $z>0$.
On the other hand, the limiting birth-and-death dynamics has only one of these measures as the symmetrizing one.  Thus, the result of the scaling  essentially depends on the initial distribution of the dynamics.
\end{remark}

 \begin{center}
{\bf Acknowledgements}\end{center}

 The authors acknowledge the financial support of the SFB 701 ``Spectral
structures and topological methods in mathematics'', Bielefeld University.

\end{document}